\documentclass[12pt]{article}
\title{\vspace{-2.5em}Unimodular Hyperbolic Triangulations: \\
   Circle Packing and Random Walk}

\author{Omer Angel \quad
   Tom Hutchcroft \quad
   Asaf Nachmias \quad
   Gourab Ray}
\date{}%{\small \today}

\AtEndDocument{
  \bigskip
  \small
  \textsc{Omer Angel, Tom Hutchcroft} \par
  \textsc{Department of Mathematics, University of British Columbia} \par
  \textit{Email:} \texttt{$\{$angel,thutch$\}$@math.ubc.ca} \par

  \medskip

  \textsc{Asaf Nachmias} \par
  \textsc{School of Mathematical Sciences, Tel Aviv University} \par
  \textit{Email:} \texttt{asafnach@post.tau.ac.il} \par

  \medskip

  \textsc{Gourab Ray} \par
  \textsc{Statistical Laboratory, University of Cambridge}\par
  \textit{Email:} \texttt{rg508@cam.ac.uk} \par
}

\usepackage{amsmath,amssymb,amsfonts,amsthm}
\usepackage[protrusion=true,expansion=true]{microtype}
\usepackage{color}

\usepackage[colorlinks=true,linkcolor=black,citecolor=black,urlcolor=black,pdfborder={0 0 0}]{hyperref}
\usepackage{graphicx}
\usepackage[font=sf, labelfont={sf,bf}, margin=1cm]{caption}
\usepackage[nameinlink]{cleveref}

\crefname{theorem}{Theorem}{Theorems}
\crefname{thm}{Theorem}{Theorems}
\crefname{lemma}{Lemma}{Lemmas}
\crefname{lem}{Lemma}{Lemmas}
\crefname{remark}{Remark}{Remarks}
\crefname{prop}{Proposition}{Propositions}
\crefname{defn}{Definition}{Definitions}
\crefname{corollary}{Corollary}{Corollaries}
\crefname{conjecture}{Conjecture}{Conjectures}
\crefname{question}{Question}{Questions}
\crefname{chapter}{Chapter}{Chapters}
\crefname{section}{Section}{Sections}
\crefname{figure}{Figure}{Figures}
\newcommand{\crefrangeconjunction}{--}

\theoremstyle{plain}
\newtheorem{thm}{Theorem}[section]
\newtheorem{lemma}[thm]{Lemma}
\newtheorem{lem}[thm]{Lemma}
\newtheorem{corollary}[thm]{Corollary}
\newtheorem{prop}[thm]{Proposition}
\newtheorem{conjecture}[thm]{Conjecture}

\theoremstyle{definition}

\newtheorem{problem}[thm]{Problem}
\theoremstyle{remark}
\newtheorem{remark}[thm]{Remark}

\numberwithin{equation}{section}

\renewcommand{\P}{\mathbb P}
\newcommand{\Z}{\mathbb Z}
\newcommand{\E}{\mathbb E}
\newcommand{\R}{\mathbb R}
\newcommand{\N}{\mathbb N}
\newcommand{\eps}{\varepsilon}

\newcommand{\note}[1]{{\color{red}{[note: #1]}}}

\usepackage{caption,subcaption}
\usepackage{bbm}
\usepackage[margin=35mm]{geometry}

\newcommand{\core}{\operatorname{core}}
\newcommand{\VEL}{\operatorname{VEL}}
\newcommand{\D}{\mathbb D}
\renewcommand{\H}{\mathbb H}
\newcommand{\eqd} {\overset{d}{=}}
\newcommand{\cC}{\mathcal C}
\newcommand{\cI}{\mathcal I}
\newcommand{\C}{\mathbb C}
\renewcommand{\bar}{\overline}

\begin{document}
\maketitle
\vspace{-2em}
\begin{abstract}
  We show that the circle packing type of a unimodular random plane
  triangulation is parabolic if and only if the expected degree of the root
  is six, if and only if the triangulation is amenable in the sense of
  Aldous and Lyons \cite{AL07}.  As a part of this, we obtain an
  alternative proof of the Benjamini-Schramm Recurrence Theorem
  \cite{BeSc}.

  Secondly, in the hyperbolic case, we prove that the random walk almost
  surely converges to a point in the unit circle, that the law of this
  limiting point has full support and no atoms, and that the unit circle is
  a realisation of the Poisson boundary.  Finally, we show that the simple
  random walk has positive speed in the hyperbolic metric.
\end{abstract}

\begin{figure}[h!]
  \centering
  \includegraphics[width=0.7\textwidth]{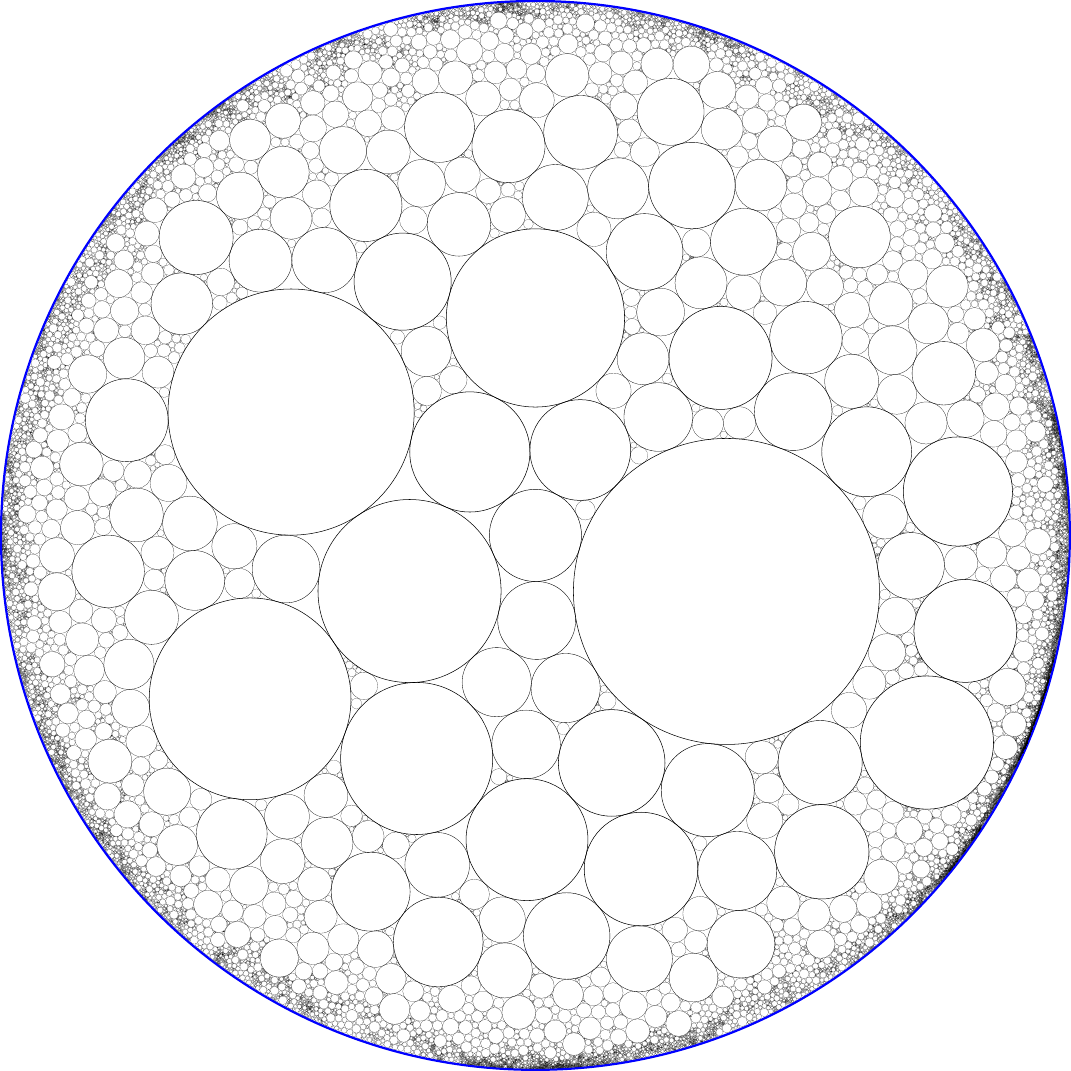}
  \caption{A circle packing of a random hyperbolic triangulation.}
\end{figure}

\thispagestyle{empty}

%%%%%%%%%%%%%%%%%%%%%%%%%%%%%%%%%%%%%%%%%%%%%%%%%%%%%%%%%%%%%%%%%%
\section{Introduction}\label{sec:intro}

A \textbf{circle packing} of a planar graph $G$ is a set of circles with
disjoint interiors in the plane, one for each vertex of $G$, such that two
circles are tangent if and only if their corresponding vertices are
adjacent in $G$.  The Koebe-Andreev-Thurston Circle Packing Theorem
\cite{K36,Th78} states that every finite simple planar graph has a circle
packing; if the graph is a triangulation (i.e.\ every face has three
sides), the packing is unique up to M\"obius transformations and
reflections.  He and Schramm \cite{HS93,HeSc} extended this theorem to infinite,
one-ended, simple triangulations, showing that each such triangulation
admits a locally finite circle packing either in the Euclidean plane or in
the hyperbolic plane (identified with the interior of the unit disc), but
not both. See \cref{subsec:CP_background} for precise details.
This result is a discrete analogue of the Uniformization Theorem,
which states that every simply connected, non-compact Riemann surface is
conformally equivalent to either the plane or the disc (indeed, there are deep
connections between circle packing and conformal maps, see
\cite{Rohde11, St05} and references therein).  Accordingly, a triangulation
is called \textbf{CP parabolic} if it can be circle packed in the plane and
\textbf{CP hyperbolic} otherwise.

Circle packing has proven instrumental in the study of random walks on
planar graphs \cite{BS96a,BeSc,HeSc,GGN13}.  For graphs with bounded
degrees, a rich theory has been established connecting the geometry of the
circle packing and the behaviour of the random walk. Most notably, a
one-ended, bounded degree triangulation is CP hyperbolic if and only if
random walk on it is transient \cite{HeSc} and in this case it is also
non-Liouville, i.e.\ admits non-constant bounded harmonic functions
\cite{BS96a}.

The goal of this work is to develop a similar, parallel theory for
\emph{random} triangulations.  Particular motivations come from the
Markovian hyperbolic triangulations constructed recently in \cite{AR13} and
\cite{PSHIT}.  These are hyperbolic variants of the UIPT \cite{UIPT1} and
are conjectured to be the local limits of uniform triangulations in high
genus.  Another example is the Poisson-Delaunay triangulation in the
hyperbolic plane, studied in \cite{BS00} and \cite{BPP14}.  All these
triangulations have unbounded degrees, rendering existing methods
ineffective (for example methods used in \cite{ABGN14,BS96a,HeSc}).

Indeed, in the absence of bounded degree the existing theory fails in many
ways.  For example, in a circle packing of a triangulation with bounded
degrees, radii of adjacent circles have uniformly bounded ratios (a fact
known as the Ring Lemma \cite{RS87}).  The absence of such a uniform
bound invalidates important resistance estimates. This is not a mere
technicality: one can add extra circles in the interstices of the circle
packing of the triangular lattice to give the random walk drift in
arbitrary directions. This does not change the circle packing type, but
allows construction of a graph that is CP parabolic but transient or even
non-Liouville. Indeed, the main effort in \cite{GGN13} was to overcome this
sole obstacle in order to prove that the UIPT is recurrent.

The hyperbolic random triangulations of \cite{PSHIT} and \cite{BPP14} make
up for having unbounded degrees by a different useful property:
\emph{unimodularity} (essentially equivalent to \emph{reversibility},
 see \cref{subsec:unimodularity,subsec:reversibility}).  This
allows us to apply probabilistic and ergodic arguments in place of the
analytic arguments appropriate to the bounded degree case.  Our first main
theorem establishes a probabilistic characterisation of the CP type for
unimodular random rooted triangulations, and connects it to the geometric
property of \emph{invariant (non-)amenability}, which we define in
\cref{subsec:invariant_expansion}.

\begin{thm}\label{thm:classification}
  Let $(G,\rho)$ be an infinite, simple, one-ended, ergodic unimodular
  random rooted planar triangulation.  Then either
  \begin{quote}
    $\E[\deg(\rho)]=6$, in which case $(G,\rho)$ is invariantly amenable
    and almost surely CP parabolic,
  \end{quote}
  or else
  \begin{quote}
    $\E[\deg(\rho)]>6$, in which case $(G,\rho)$ is invariantly
    non-amenable and almost surely CP hyperbolic.
  \end{quote}
\end{thm}

This theorem can be viewed as a local-to-global principle for unimodular
triangulations. That is, it allows us to identify the circle packing type
and invariant amenability, both global properties, by calculating the
expected degree, a very local quantity.  For example, if $(G,\rho)$ is a
simple, one-ended triangulation that is obtained as a local limit of planar
graphs, then by Euler's formula and Fatou's lemma its average degree is at
most $6$, so that \cref{thm:classification} implies it is almost surely CP
parabolic. If in addition $(G,\rho)$ has bounded degrees, then it is
recurrent by He-Schramm \cite{HeSc}.  In particular, this gives an
alternative proof of the Benjamini-Schramm Recurrence Theorem \cite{BeSc}
in the primary case of a one-ended limit. We handle the remaining cases in
\cref{sec:BSproof}.  Unlike the
proof of \cite{BeSc}, whose main ingredient is a quantitative estimate for
finite circle packings \cite[Lemma 2.3]{BeSc},
our method works with infinite triangulations directly and implies the
following generalisation:

\begin{prop}\label{P:one_ended}
  Any unimodular, simple, one-ended random rooted planar triangulation
  $(G,\rho)$ with bounded degrees and $\E[\deg(\rho)]=6$ is almost surely
  recurrent.
\end{prop}

This trivially extends the Benjamini-Schramm result, since any local limit
of finite planar graphs is unimodular.  An important open question is
whether every unimodular random graph is a Benjamini-Schramm limit of
finite graphs.  In a forthcoming paper \cite{unimodular2}, we show that any
unimodular planar graph $G$ is a limit of some sequence of finite graphs
$G_n$, and that if $G$ is a triangulation with $\E[\deg(\rho)]=6$ then
$G_n$ can also be taken to be planar.  In particular, any graph to which
\cref{P:one_ended} applies is also a local limit of finite planar graphs
with bounded degrees.  Consequently there are no graphs to which this
result applies and the Benjamini-Schramm Theorem does not.  Note however,
that for a given unimodular planar triangulation, it may not be obvious how
to find this sequence of graphs.  We remark that the dichotomy of
\cref{thm:classification} has many extensions, applying to more general
maps and holding further properties equivalent. We address these in
\cite{unimodular2}. See \cite{BeSt91} and \cite[Theorem 10.2]{HeSc} for
earlier connections between the CP type and degree distributions in the
deterministic setting.

Our method of proof relies on the deep theorem of Schramm \cite{Schramm91}
that the circle packing of a triangulation in the disc or the plane is
unique up to M\"obius transformations fixing the disc or the plane as
appropriate. We use this fact throughout the paper in an essential way: it
implies that any quantity derived from the circle packing in the disc or
the plane that is invariant to M\"obius transformations is determined by
the graph $G$ and not by our choice of circle packing.  Key examples of
such quantities are
angles between adjacent edges in the associated drawings with hyperbolic or
Euclidean geodesics (see \cref{sec:type}), hyperbolic radii of circles in
the hyperbolic case, and ratios of Euclidean radii in the parabolic case.

\paragraph{Boundary Theory.}
Throughout, we realize the hyperbolic plane as the Poincar\'e disc
$\{|z|<1\}$ with metric $d_{hyp}$.  The unit circle $\{|z|=1\}$ is the
boundary of the hyperbolic plane in several geometric and probabilistic
senses.  For a general graph embedded in the hyperbolic plane, the unit
circle may or may not coincide with probabilistic notions of the graph's
boundary.

When a bounded degree triangulation is circle packed in the disc, Benjamini
and Schramm \cite{BS96a} showed that the random walk converges to a point
in the circle almost surely and that the law of the limit point has full
support and no atoms.  More recently, it was shown by the first and third
authors together with Barlow and Gurel-Gurevich \cite{ABGN14} that the unit
circle is a realisation of both the Poisson and Martin boundaries of the
triangulation. Similar results regarding square tiling were obtained in
\cite{BS96b} and \cite{G13}.

Again, these theorems fail for some triangulations with unbounded degrees.
Starting with any CP hyperbolic triangulation, one can add circles in the
interstices of the packing so as to create drifts along arbitrary paths.
  In this way, one can force the random walk to spiral in the unit
disc and not converge to any point in the boundary. One can also create a
graph for which the walk can converge to a single boundary point from two or
more different angles each with positive probability, so that the exit
measure is atomic and the unit circle is no longer a realisation of the
Poisson boundary.  Our next result recovers the boundary theory in the
unimodular setting.

When $\cC$ is a circle packing of a graph $G$ in the disc $\D$, we write
$\cC=(z,r)$ where $z(v)$ is the (Euclidean) centre of the circle
corresponding to $v$, and $r(v)$ is its Euclidean radius.  Recall that the
hyperbolic metric on the unit disc is defined by
\[
|d_{hyp}(z)| = \frac{2|dz|}{1-|z|^2},
\]
and that circles in the Euclidean metric are also hyperbolic circles (with
different centres and radii).  We write $z_h(v)$ and $r_h(v)$ for the
hyperbolic centre and radius of the circle corresponding to $v$.  We use
$P_v^G$ and $E^G_v$ to
denote the probability and expectation (conditioned on $G$) with respect to
random walk $(X_n)_{n \geq 0}$ on $G$ started from a vertex $v$.

\begin{thm}\label{thm:boundary}
  Let $(G,\rho)$ be a simple, one-ended, CP hyperbolic unimodular random
  planar triangulation with $\E[\deg^2(\rho)]<\infty$. Let $\cC$ be a
  circle packing of $G$ in the unit disc, and let $(X_n)$ be a simple
  random walk on $G$.  The following hold conditional on $(G,\rho)$ almost
  surely:
  \begin{enumerate}
  \item\label{thm:convergence} $z(X_n)$ and $z_h(X_n)$ both converge to a
    (random) point denoted $\Xi \in\partial \D$,
  \item\label{thm:exitmeasure} The law of $\Xi$ has full support
    $\partial \D$ and no atoms.
  \item\label{thm:Poisson} $\partial \D$ is a realisation of the Poisson
    boundary of $G$.  That is, for every bounded harmonic function $h$ on
    $G$ there exists a bounded measurable function $g: \partial \D \to \R$
    such that
    \[
    h(v) = E^G_v [ g(\Xi)].
    \]
  \end{enumerate}
\end{thm}

We refer to the law of $\Xi$ conditional on $(G,\rho)$ as the \textbf{exit
  measure} from $v$.  In \cref{sec:extensions} we extend this result to
weighted and non-simple triangulations, with the obvious changes.  One
ingredient in the proof of the absence of atoms is a more general
observation, \cref{lem:atom}, which states roughly that exit measures on
boundaries of stationary graphs are either non-atomic or trivial almost
surely.

Our final result relates exponential decay of the Euclidean radii along the
random walk to speed in the hyperbolic metric.

\begin{thm}\label{thm:exp_decay}
  Let $(G,\rho)$ be a simple, one-ended, CP hyperbolic unimodular random
  rooted planar triangulation with $\E[\deg^2(\rho)]<\infty$ and let $\cC$
  be a circle packing of $G$ in the unit disc.  Then almost surely
  \[
  \lim_{n \to \infty} \frac{d_{hyp}(z_h(\rho),z_h(X_n))}{n}
  = \lim_{n \to \infty} \frac{-\log r(X_n)}{n}
  > 0.
  \]
  In particular, both limits exist.  Moreover, the limits do not depend on
  the choice of packing, and if $(G,\rho)$ is ergodic then this limit is an
  almost sure constant.
\end{thm}

Thus the random walk $(X_n)$ has positive asymptotic speed in the
hyperbolic metric, the Euclidean radii along the walk decay exponentially,
and the two rates agree.

%\tableofcontents

\paragraph{Organization of the paper.}
In \cref{sec:examples} we review the motivating examples of unimodular
hyperbolic random triangulations to which our results apply.
In \cref{subsec:unimodularity} and \cref{subsec:reversibility} we give
background on unimodularity, reversibility and related topics.
In \cref{subsec:invariant_expansion} we recall Aldous and Lyons's notion of
invariant amenability \cite{AL07} and prove one of its important consequences.
In \cref{subsec:CP_background} we recall the required results on circle
packing and discuss measurability.
\cref{sec:type} contains the proof of \cref{thm:classification} as well as
a discussion of how to handle the remaining (easier) cases of the
Benjamini-Schramm Theorem.
\cref{thm:boundary} is proved in \cref{sec:boundarytheory} and
\cref{thm:exp_decay} is proved in \cref{sec:exp_decay}.
Background on the Poisson boundary is provided before the proof of
\cref{thm:boundary}(3) in \cref{subsec:poisson}.
In \cref{sec:extensions} we discuss extensions of our results to non-simple
and weighted triangulations.
We end with some open problems in \cref{sec:problems}.

%%%%%%%%%%%%%%%%%%%%%%%%%%%%%%%%%%%%%%%%%%%%%%%%%%%%%%%%%%%%%%%%%%%%%%%%
\section{Examples}
\label{sec:examples}

Benjamini-Schramm limits of random maps have been objects of great interest
in recent years, serving as discrete models of $2$-dimensional quantum
gravity.  Roughly, the idea is to consider a uniformly random map from some
class of rooted maps (e.g.\ all triangulations or quadrangulations of the
sphere of size $n$) and take a local limit as the size of the maps tends to
infinity.  The first such construction was the UIPT \cite{UIPT1}; see
also \cite{UIPT2,BC11,CM12,AC13,Unilocal13}.

\paragraph{Curien's PSHT.}
Recently, hyperbolic versions of the UIPT and related maps have been
constructed: half-plane versions in \cite{AR13} and full-plane versions in
\cite{PSHIT}.  These are constructed directly, and are believed but not yet
known to be the limits of finite maps (see below).  The full plane
triangulations form a one (continuous) parameter family $\{T_\kappa
\}_{\kappa\in (0,2/27)}$ (known as the PSHT, for \emph{Planar Stochastic
  Hyperbolic Triangulation}).  They are reversible and ergodic, have
anchored expansion and are therefore invariantly non-amenable.  The degree
of the root in $T_\kappa$ is known to have an exponential tail, so that
all of its moments are finite.
% Applying \cref{thm:coreboundary}, we obtain a characterisation of the
% Poisson boundary of $T_\kappa$ as the geometric boundary of the circle
% packing of $\overline{\core}(T_\kappa)$ in $\D$ (see
% \cref{sec:extensions}).
These triangulations are not simple, so our main results do not apply to
them directly, but by considering their \emph{simple cores} we are still
able to obtain a geometric representation of their Poisson boundary
(see \cref{sec:extensions}).

\paragraph{Benjamini-Schramm limits of maps in high genus.}
It is conjectured that the PSHT $T_\kappa$ is the Benjamini-Schramm limit
of the uniform triangulation with $n$ vertices of a surface of genus
$\lfloor \theta n\rfloor$, for some $\theta=\theta(\kappa)$ (see e.g.\
\cite{Rayphd14} for precise definitions of maps on general surfaces). In
our upcoming paper \cite{unimodular2} we prove that all one-ended
unimodular random rooted planar triangulations are also Benjamini-Schramm
limits of finite triangulations. If the triangulation has expected degree
greater than 6, then the finite approximating triangulations necessarily
have genus linear in their size.

In the context of circle packing, it may be particularly interesting to
take the Benjamini-Schramm limit $(T,\rho)$ of the uniform \emph{simple}
triangulation with $n$ vertices of the $\lfloor \theta n \rfloor$-holed
torus $T_n$.  This limit (which we conjecture exists) should be a simple
variant of the PSHT.  Letting $\rho_n$ be a uniformly chosen root of $T_n$,
it should also be the case that $\E[\deg(\rho_n)] \to \E[\deg(\rho)] > 6$
and $\E[\deg(\rho)^2]<\infty$, so that our results would be applicable to
the circle packing of $(T,\rho)$.

\paragraph{Delaunay triangulations of the hyperbolic plane.}
Start with a Poisson point process in the hyperbolic plane with intensity
$\lambda$ times the hyperbolic area measure, and add a root point at the
origin.  Consider now the \textbf{Delaunay triangulation} with this point
process as its vertex set, where three vertices $u,v,w$ form a triangle if
the circle through $u,v,w$ contains no other points of the process.  This
triangulation, known as the Poisson-Delaunay triangulation, is naturally embedded in the hyperbolic plane with hyperbolic
geodesic edges.  These triangulations, studied in \cite{BS00,BPP14}, are
unimodular when rooted at the point at the origin.  They are known to have
anchored expansion \cite{BPP14} and are therefore invariantly non-amenable.
(We also get a new
proof of non-amenability from \cref{thm:classification}, as one can show
the expected degree to be greater than six by transporting angles as in
the proof of \cref{thm:classification}.)  The Poisson-Delaunay
triangulations are also simple and one-ended, and the degree of the root has
finite second moment, so that our results apply directly to their circle
packings.

%%%%%%%%%%%%%%%%%%%%%%%%%%%%%%%%%%%%%%%%%%%%%%%%%%%%%%%%%%%%%%%%%%%%%%%
\section{Background and Definitions}

\subsection{Unimodular random graphs and maps}
\label{subsec:unimodularity}

Unimodularity of graphs (both fixed and random) has proven to be a useful
and natural property in a number of settings.  We give here the required
definitions and some of their consequences, and refer the reader to
\cite{AL07,LP:book} for further background.

A \textbf{rooted graph} $(G,\rho)$ is a graph $G=(V,E)$ with a
distinguished vertex $\rho$ called the \textbf{root}. We will allow our
graphs to contain self-loops and multiple edges, and refer to graphs
without either as \textbf{simple}. A graph is said to be \textbf{one-ended} if
the removal of any finite set of vertices leaves precisely one infinite
connected component. A graph isomorphism between two rooted graphs is a
rooted graph isomorphism if it preserves the root.

A \textbf{map} is a proper (that is, with non-intersecting edges) embedding
of a connected graph into a surface, viewed up to orientation preserving
homeomorphisms of the surface, so that all connected components of the
complement (called \textbf{faces}) are topological discs.\footnote{There is
  an additional constraint regarding boundaries of faces of infinite
  degree.  However, this condition is automatically satisfied for
  triangulations and for simply connected maps, so that we need not worry
  about it in this paper.} The map is \textbf{planar} if the surface is
homeomorphic to an open subset of the sphere, and is \textbf{simply
  connected} if the surface is homeomorphic to the sphere or the plane.  A
map is a \textbf{triangulation} if every face is incident to exactly three
edges.  Note that an infinite planar triangulation is simply connected if and only if
it is one-ended.

Every connected graph $G$ can be made into a metric space by endowing it with
the shortest path metric $d_G$.  By abuse of notation, we use the ball
$B_n(G,u)$ to refer both to the set of vertices $\{v \in V: d_G(u,v) \leq
n\}$ and the induced subgraph on this set, rooted at $u$. The balls in a
map inherit a map structure from the full map.

The \textbf{local topology} on the space of rooted connected graphs
(introduced in \cite{BeSc}) is the topology induced by the metric
\[
d_{\mathrm{loc}}\big((G,\rho),(G',\rho')\big) = e^{-R}
\quad \text{ where } \quad
R = \sup \big\{ n\geq 0 : B_n(G,\rho) \cong B_n(G',\rho')\big\}.
\]
The local topology on rooted maps is defined similarly by requiring the
isomorphism of the balls to be an isomorphism of rooted maps.  We denote by
$\mathcal{G}_\bullet$ and $\mathcal{M}_\bullet$ the spaces of isomorphism
classes of rooted connected graphs and of maps with their respective local
topologies.  \textbf{Random rooted graphs and maps} are Borel random
variables taking values in these spaces.

Several variants of these spaces will also be of use.  A (countably)
\textbf{marked} graph is a graph together with a mark function $m:V \cup E
\to M$ which gives every edge and vertex a mark in some countable set
$M$. A graph isomorphism between marked graphs is an isomorphism of
marked graphs if it preserves the marks.  The local topologies on rooted
marked graphs and maps is defined in the obvious way.  These spaces are
denoted $\mathcal{G}^M_\bullet$ and $\mathcal{M}^M_\bullet$.  Sometimes we
will consider maps with marks only on vertices or only on edges; these fit
easily into our framework. Marked graphs are special cases of what Aldous
and Lyons \cite{AL07} call \textbf{networks}, for which the marks may take
values in any separable complete metric space.

Similarly, we define $\mathcal{G}_{\bullet\bullet}$ (resp.\
$\mathcal{M}_{\bullet\bullet}$) to be the spaces of doubly rooted (that is,
with a distinguished ordered pair of vertices) connected graphs (resp.\ maps)
$(G,u,v)$.  These spaces, along with their marked versions, are equipped
with natural variants of the local topology.  All such spaces we consider
are Polish.

A \textbf{mass transport} is a non-negative Borel function $f :
\mathcal{G}_{\bullet \bullet} \to \R_+$.  A random rooted graph
$(G,\rho)$ is said to be \textbf{unimodular} if it satisfies the
\textbf{mass transport principle}: for any mass transport $f$,
\[
\E\Bigg[\sum_{v \in V(G)} f(G,\rho,v)\Bigg] =  \E\Bigg[\sum_{v \in V(G)}
  f(G,v,\rho)\Bigg].
\]
In other words,
\begin{center}
  `Expected mass out equals expected mass in.'
\end{center}
This definition generalises naturally to define unimodular marked graphs
and maps. Importantly, any finite graph $G$ with a uniformly chosen root
vertex $\rho$ satisfies the mass transport principle.

The laws of unimodular random rooted graphs form a weakly closed, convex
subset of the space of probability measures on $\mathcal{G}_\bullet$, so
that weak limits of unimodular random graphs are unimodular. In particular,
a weak limit of finite graphs with uniformly chosen roots is unimodular:
such a limit of finite graphs is referred to as a \textbf{Benjamini-Schramm
  limit}.  It is a major open problem to determine whether all unimodular
random rooted graphs arise as Benjamini-Schramm limits of finite graphs
\cite[\S 10]{AL07}. As mentioned in \cref{sec:examples}, we provide a
positive solution to this problem in the planar case in the upcoming work
\cite{unimodular2}, proving that every simply connected unimodular random
rooted planar map is a Benjamini-Schramm limit of finite maps.

A common use of the mass transport principle to obtain proofs by
contradiction is the following. If $(G,\rho)$ is a unimodular random rooted
graph and $f$ is a mass transport such that the mass sent out from each
vertex $\sum_vf(G,u,v) \leq M$ is uniformly bounded almost surely, then
almost surely there are no vertices that receive infinite mass: if vertices
receiving infinite mass were to exist with positive probability, the root
would be such a vertex with positive probability \cite[Lemma 2.3]{AL07},
contradicting the mass transport principle.

\subsection{Random walk, reversibility and ergodicity}
\label{subsec:reversibility}

Recall that the \textbf{simple random walk} on a graph is the Markov chain
that chooses $X_{n+1}$ from among the neighbours of $X_n$ weighted by the number of shared edges.
Define $\mathcal{G}_\leftrightarrow$ (resp.\ $\mathcal{M}_\leftrightarrow$)
to be spaces of isomorphism classes of graphs (resp.\ maps) equipped with a
bi-infinite path $(G,(x_n)_{n \in \Z})$, which we endow with a natural variant of the local topology.  When $(G,\rho)$ is a random graph
or map, we let $(X_n)_{n\geq 0}$ and $(X_{-n})_{n\geq 0}$ be two
independent simple random walks started from $\rho$ and consider
$(G,(X_n)_{n \in \Z})$ to be a random element of
$\mathcal{G}_\leftrightarrow$ or $\mathcal{M}_\leftrightarrow$ as
appropriate.

A random rooted graph $(G,\rho)$ is \textbf{stationary} if $(G,\rho) \eqd (G,X_1)$ and \textbf{reversible} if $(G,\rho,X_1)
\eqd (G,X_1,\rho)$ as doubly rooted graphs.  Equivalently, $(G,\rho)$ is
reversible if and only if $(G,(X_n)_{n \in \Z})$ is stationary with respect
to the shift:
\[
(G,(X_n)_{n \in \Z}) \eqd (G,(X_{n+k})_{n \in \Z}) \text{ for every $k \in
  \Z$.}
\]
To see this, it suffices to prove
that if $(G,\rho)$ is a reversible random graph then
$(X_1,\rho,X_{-1},X_{-2},\ldots)$ has the law of a simple random walk
started from $X_1$. But $(\rho,X_{-1},\ldots)$ is a simple random walk
started from $\rho$ independent of $X_1$ and, conditional on $(G,X_1)$,
reversibility implies that $\rho$ is uniformly distributed among the
neighbours of $X_1$, so $(X_1,\rho,X_{-1},X_{-2},\ldots)$ has the law of a
simple random walk as desired.

We remark that if $(G,\rho)$ is stationary but not necessarily reversible,
it is still possible to extend the walk to a doubly infinite path
$(X_n)_{n\in\Z}$ so that $G$ is stationary along the path.  The difference
is that in the reversible case the past $(X_n)_{n\leq 0}$ is itself a
simple random walk with the same law as the future.

Reversibility is related to unimodularity via the following bijection, which is implicit in \cite{AL07} and proven explicitly in \cite{BC2011}: if
$(G,\rho)$ is reversible, then biasing by $\deg(\rho)^{-1}$ (i.e.\
reweighing the law of $(G,\rho)$ by the Radon-Nikodym derivative
$\deg(\rho)^{-1}/\E[\deg(\rho)^{-1}]$) gives an equivalent unimodular
random rooted graph, and conversely if $(G,\rho)$ is a unimodular random rooted graph
with finite expected degree, then biasing by $\deg(\rho)$ gives an
equivalent reversible random rooted graph.  Thus, the laws of reversible  random rooted graphs are in bijection with the laws of unimodular random rooted graphs
 for which the root degree has finite expectation.

An event $A\subset \mathcal{G}_\leftrightarrow$ is said to be
\textbf{invariant} if $(G,(X_{n})_{n \in \Z})\in A$ implies
$(G,(X_{n+k})_{n \in \Z}) \in A$ for each $k \in \Z$.  A reversible or unimodular random
graph is said to be \textbf{ergodic} if the law of $(G,(X_{n})_{n \in \Z})$
gives each invariant event probability either zero or one. An event $A \subseteq \mathcal{G}_\bullet$ is \textbf{rerooting-invariant} if $(G,\rho)\in A$ implies $(G,v)\in A$ for every vertex $v$ of $G$.

\begin{thm}[Characterisation of ergodicity {\cite[\S 4]{AL07}}]
  Let $(G,\rho)$ be a unimodular  random rooted graph with $\E[\deg(\rho)]<\infty$ (resp. a reversible random rooted graph). The following are
  equivalent.
  \begin{enumerate}\itemsep0em
  \item $(G,\rho)$ is ergodic.
  \item Every rerooting-invariant event $A\subseteq \mathcal{G}_\bullet$ has probability in $\{0,1\}$.
  \item The law of $(G,\rho)$ is an extreme point of the weakly closed
    convex set of laws of unimodular (resp. reversible) random rooted graphs.
\end{enumerate}
  %Additionally, when $(G,\rho)$ is reversible, these are equivalent to $(G,\rho)$ being ergodic.
\end{thm}
(The equivalence of items $2$ and $3$ holds for unimodular random rooted graphs without the assumption of finite expected degree.)
A consequence of the extremal characterisation is that every unimodular
random rooted graph is a \emph{mixture} of ergodic unimodular random rooted
graphs, meaning that it may be sampled by first sampling a random law of an
ergodic unimodular random rooted graph, and then sampling from this
randomly chosen law - this is known as an ergodic decomposition and its
existence is a consequence of Choquet's Theorem. In particular, whenever we want to prove that a unimodular random rooted
graph with some almost sure property also has some other almost sure
property, it suffices to consider the ergodic case. The same comment applies for reversible random rooted graphs.

\subsection{Invariant amenability}\label{subsec:invariant_expansion}

We begin with a brief review of general amenability, before combining it
with unimodularity for the notion of invariant amenability. We refer the
reader to \cite[\S 6]{LP:book} for further details on amenability in
general, and \cite[\S 8]{AL07} for invariant amenability.

A \textbf{weighted} graph is a graph together with a weight function $w : E
\rightarrow \R_+$.  Unweighted multigraphs may always be considered as
weighted graphs by setting $w\equiv 1$.  The weight function is extended to
vertices by $w(x) = \sum_{e\ni x} w(e)$, and (with a slight abuse of
notation) to sets of edges or vertices by additivity.  The simple random
walk $X=(X_n)_{n \geq0 }$ on a weighted graph is the Markov chain on $V$
with transition probabilities $p(x,y) = w(x,y)/w(x)$.  Here, our graphs are
allowed to have infinite degree provided $w(v)$ is finite for every vertex.

The (edge) \textbf{Cheeger constant} of an infinite weighted graph is
defined to be
\[
\mathbf{i}_E(G) = \inf \left\{\frac{w(\partial_E W)}{w(W)} : \emptyset
  \neq W \subset V \text{ finite} \right\}
\]
where $\partial_E W$ denotes the set of edges with exactly one end in $W$.
A graph is said to be \textbf{amenable} if its Cheeger constant is zero and
\textbf{non-amenable} if it is positive.

The Markov operator associated to simple random walk on $G$ is the bounded,
self-adjoint operator from $L^2(V,w)$ to itself defined by $(Pf)(u) = \sum
p(u,v)f(v)$. The norm of this operator is commonly known as the
\textbf{spectral radius} of the graph.  If $u,v \in V$ then the transition
probabilities are given by $p_n(u,v) = \left\langle P^n \mathbbm{1}_v,
  \mathbbm{1}_u/w(u) \right\rangle_w$, so that, by Cauchy-Schwarz,
\begin{equation}\label{eq:CS}
  p_n(u,v) \leq \sqrt{\frac{w(v)}{w(u)}} \|P\|_w^n
\end{equation}
and in fact $\|P\|_w = \limsup_{n \to \infty}p_{n}(u,v)^{1/n}$.  A
fundamental result, originally proved for Cayley graphs by Kesten
\cite{Kes59}, is that the spectral radius of a weighted graph is less than
one if and only if the graph is non-amenable (see
\cite[Theorem~6.7]{LP:book} for a modern account).
%This was made quantitative by Cheeger who proved the inequality
%\begin{equation}
%  \|P\|_\omega \leq 1-\frac{\mathbf{i}_E(G)^2}{2}. \label{eq:Cheeger}
%\end{equation}
As an immediate consequence, non-amenable graphs are transient for simple
random walk.

\subsubsection{Invariant amenability}

There are natural notions of amenability and expansion for unimodular
random networks due to Aldous and Lyons \cite{AL07}. A
\textbf{percolation} on a unimodular random rooted graph $(G,\rho)$ is a
random assignment of $\omega : E \cup V \to \{0,1\}$ such that the marked graph
$(G,\rho,\omega)$ is unimodular.  We think of $\omega$ as a random subgraph
of $G$ consisting of the `open' edges and vertices $\omega(e)=1$,
$\omega(v)=1$, and may assume without loss of generality that if an edge is
open then so are both of its endpoints.
%A \textbf{bond percolation} is a
%percolation for which all vertices have $\omega(v)=1$ and only bonds
%(edges) may be missing.

%A \textbf{percolation} on a unimodular random rooted graph $(G,\rho)$ is a
%random assignment of $\omega : E \to \{0,1\}$ such that the marked graph
%$(G,\rho,\omega)$ is unimodular.  We think of $\omega$ as a random subgraph
%of $G$ consisting of the ``open'' edges $\omega(e)=1$.

%  A \textbf{bond percolation} is a
%percolation for which all vertices have $\omega(v)=1$ and only bonds
%(edges) may be missing.

The \textbf{cluster} $K_\omega(v)$ at a vertex $v$ is the connected
component of $v$ in $\omega$, i.e.\ the set of vertices for which there is
a path of open edges to $v$ (by convention, if $\omega(v)=0$ or if there are no open edges touching $v$ we put $K_\omega(v) = \{v\}$). A percolation is said to be \textbf{finitary} if all of its clusters are finite almost surely.  The \textbf{invariant Cheeger constant}
of an ergodic unimodular random rooted graph $(G,\rho)$ is defined to be
\begin{equation}
  \mathbf{i}^{\mathrm{inv}}((G,\rho))= \inf \left\{ \E\left[
      \frac{|\partial_E K_\omega(\rho)|}{|K_\omega(\rho)|}\right] : \omega
    \text{ a finitary percolation on $(G,\rho)$} \right\}.
\end{equation}
%(If $\omega(\rho)=0$ we use the convention that $K_\omega(\rho)=\{\rho\}$.)

The invariant Cheeger constant is closely related to another quantity: mean
degrees in finitary percolations. Let $\deg_\omega(\rho)$ denote the degree
of $\rho$ in $\omega$ (seen as a subgraph; if $\rho \not \in \omega$ we set
$\deg_\omega(\rho)=0$) and let
\[
\alpha((G,\rho)) = \sup \left\{ \E\left[\deg_\omega(\rho)\right] :
  \omega \text{ a finitary percolation on $(G,\rho)$}
\right\}.
\]
An easy application of the mass transport principle \cite[Lemma 8.2]{AL07}
shows that, for any finitary percolation $\omega$,
\[
  \E[\deg_\omega(\rho)] = \E\left[\frac{\sum_{v \in K_\omega(\rho)}
      \deg_\omega(v)}{|K_\omega(\rho)|}\right].
\]
It follows that
\[
   \E[\deg(\rho)] = \mathbf{i}^{\mathrm{inv}}((G,\rho)) + \alpha((G,\rho))
\]
so that if $\E[\deg(\rho)]<\infty$ then $\mathbf{i}^{\mathrm{inv}}((G,\rho))$ is positive if and only if
$\alpha((G,\rho))$ is strictly smaller than $\E[\deg(\rho)]$.

We say that an ergodic unimodular random rooted graph $(G,\rho)$ is
\textbf{invariantly amenable} if $\mathbf{i}^{\mathrm{inv}}((G,\rho))=0$ and
\textbf{invariantly non-amenable} otherwise. Note that this is a property
of the \emph{law} of $(G,\rho)$ and not of an individual graph.  We remark
that what we are calling invariant amenability was called amenability when
it was introduced by Aldous and Lyons \cite[\S 8]{AL07}.  We qualify it as invariant to
distinguish it from the more classical notion, which we also use below.
While any invariantly amenable graph is trivially amenable, the converse is
generally false.  An example is a $3$-regular tree where each edge is
replaced by a path of independent length with unbounded distribution; see
\cite{AL07} for a more detailed discussion.

An important property of invariantly non-amenable graphs was first proved
for Cayley graphs by Benjamini, Lyons and Schramm \cite{BLS99}.  Aldous and
Lyons \cite{AL07} noted that the proof carried through with minor
modifications to the case of invariantly non-amenable unimodular random
rooted graphs, but did not provide a proof. As this property is crucial to
our arguments, we provide a proof for completeness, which the reader may
wish to skip. When $(G,\rho)$ is an ergodic unimodular random rooted graph,
we say that a percolation $\omega$ on $G$ is \textbf{ergodic} if
$(G,\rho,\omega)$ is ergodic as a unimodular random rooted marked graph.
The following is stated slightly differently from both Theorem~3.2 in
\cite{BLS99} and Theorem~8.13 in \cite{AL07}.

\begin{thm}\label{thm:BLS}
  Let $(G,\rho)$ be an invariantly non-amenable ergodic unimodular random
  rooted graph with $\E[\deg(\rho)]<\infty$.  Then $G$ admits an ergodic
  percolation $\omega$ so that $\mathbf{i}_{E}(\omega) > 0$ and
  vertices in $\omega$ have uniformly bounded degrees in $G$.
\end{thm}

Let us stress that the condition of uniformly bounded degrees is for the
degrees in the full graph $G$, and not the degrees in the percolation.

\begin{remark}
  This theorem plays the same role for invariant non-amenability as
  Vir\'ag's oceans and islands construction \cite{V00} does for anchored
  expansion \cite{V00,ANR14,BPP14}. In particular, it gives us a
  percolation $\omega$ such that the induced network $\bar \omega$ is
  non-amenable (see the proof of \cref{lem:exp_decay}).
\end{remark}

\begin{proof}
  Let $\omega_0$ be the percolation induced by vertices of $G$ of degree at
  most $M$ and the edges connecting any two such vertices.  By monotone convergence, and since $\alpha((G,\rho)) < \E
  \deg(\rho)$, we can take $M$ to be large enough that $\E
  \deg_{\omega_0}(\rho) > \alpha((G,\rho))$.  This gives a percolation with
  bounded degrees. We shall modify it further to get non-amenability as
  follows.  Fix $\delta>0$ by
  \[
  3\delta = \E[\deg_{\omega_0}(\rho)] - \alpha((G,\rho)).
  \]

  Construct inductively a decreasing sequence of site percolations
  $\omega_n$ as follows. Given $\omega_n$, let $\eta_n$ be independent
  Bernoulli(1/2)
  site percolations on $\omega_n$, and for each set of vertices $W$ let
  $\partial^{\omega_n}_E W$ denote the set of edges of $\omega_n$ in the
  boundary of $W$.  If $K$ is a finite connected cluster of
  $\eta_n$, with small boundary in $\omega_n$, we remove it to construct
  $\omega_{n+1}$.  More precisely, let $\omega_{n+1} = \omega_n \setminus
  \gamma_n$, where $\gamma_n$ is the subgraph of $\omega_n$ induced by the
  vertex set
  \[
  \bigcup \big\{ K : \text{ $K$ a finite cluster in $\eta_n$
    with $|\partial_E^{\omega_n}(K)| < \delta|K|$} \big\} .
  \]

  Let $\omega = \cap \omega_n$ be the limit percolation, which is clearly
  ergodic.  We shall show below that $\omega\neq\emptyset$.  Any finite
  connected set in $\omega$ appears as a connected cluster in $\eta_n$ for
  infinitely many $n$.  If such a set $S$ has $|\partial_E^{\omega} S| <
  \delta|S|$ then it would have been removed at some step, and so $\omega$
  has $|\partial_E^{\omega} S| \geq \delta|S|$ for all finite connected
  $S$. Since degrees are bounded by $M$, this implies
  $\mathbf{i}_E(\omega)\geq \delta/M>0$.

  It remains to show that $\omega\neq\emptyset$.  For some $n$, and any
  vertex $u$, let $K(u)$ be its cluster in $\eta_n$.  Consider the mass
  transport
  \[
  f_n(u,v) = \begin{cases}
    \deg_{\omega_n}(v)/|K(u)| & u \in \gamma_n \text{ and } v\in K(u), \\
    E(v,K(u))/|K(u)| & u \in \gamma_n \text{ and } v\in\omega_n \setminus
    K(u), \\
    0 & u \notin \gamma_n.
  \end{cases}
  \]
  Here $E(v,K(u))$ is the number of edges between $v$ and $K(u)$.  We have
  that the total mass into $v$ is the difference $\deg_{\omega_n}(v) -
  \deg_{\omega_{n+1}}(v)$ (where the degree is $0$ for vertices not in the
  percolation) while the mass sent from a vertex $v \in \gamma_n$ is twice
  the number of edges with either end in $K(v)$, divided by $|K(v)|$.
  Applying the mass transport principle we get
  \begin{equation}
    \label{eq:MTP1}
    \E[ \deg_{\omega_n}(\rho) - \deg_{\omega_{n+1}}(\rho)] =
    \E\left[ \frac{\sum_{v\in K(\rho)} \deg_{\gamma_n}(v) +
      2|\partial_E^{\omega_n}(K(\rho))|} {|K(\rho)|}
    \mathbbm{1}_{\rho \in \gamma_n} \right].
  \end{equation}

  By a second transport, of $\deg_{\gamma_n}(u)/|K(u)|$ from every $u \in
  \gamma_n$ to each $v \in K(u)$, we see that
  \begin{equation}
    \label{eq:MTP2}
    \E\left[ \frac{\sum_{v\in K(\rho)} \deg_{\gamma_n}(v)} {|K(\rho)|}
    \mathbbm{1}_{\rho \in \gamma_n} \right] = \E[\deg_{\gamma_n}(\rho)].
  \end{equation}
  Additionally, on the event $\{\rho \in \gamma_n\}$, we have by definition
  that
  \[
  |\partial_E^{\omega_n}(K(\rho))|/|K(\rho)| \leq \delta.
  \]
  Plugging these two in \eqref{eq:MTP1} gives
  \[
  \E [ \deg_{\omega_n}(\rho)-\deg_{\omega_{n+1}}(\rho)]
  \leq \E[\deg_{\gamma_n}(\rho)] + 2\delta \P(\rho \in \gamma_n).
  \]

  Let $\gamma = \cup_{n \geq 1} \gamma_n$, which is a percolation since it
  is defined as a measurable, automorphism invariant function of $(G,\rho)$
  and the i.i.d.\ sequence of Bernoulli percolations $(\eta_n)$.  Note the
  percolations $\gamma_n$ are disjoint, so that the event $\rho \in
  \gamma_n$ can occur for at most one $n$ and that $\gamma$ is a finitary
  percolation.  Thus $\E[\deg_\gamma(\rho)] = \sum_n
  \E[\deg_{\gamma_n}(\rho)] \leq
  \alpha((G,\rho))$.  Also, $\sum_n \P(\rho \in \gamma_n) \leq 1$. Summing over
  $n$ gives
  \[
  \E[ \deg_{\omega_0}(\rho)-\deg_{\omega}(\rho)] \leq \alpha((G,\rho)) + 2\delta.
  \]
  The definition of $\delta$ leaves $\E[\deg_{\omega}(\rho)] \geq \delta$.
  Thus $\omega$ is indeed non-empty as claimed, completing the proof.
\end{proof}

\subsection{Circle packings and vertex extremal length}
\label{subsec:CP_background}

Recall that a \textbf{circle packing} $\cC$ is a collection of discs of
disjoint interior in the plane $\C$.  Given a circle packing $\cC$, we define
its \textbf{tangency map} as the map whose embedded vertex set $V$
corresponds to the centres of circles in $\cC$ and whose edges are given by
straight lines between the centres of tangent circles.  If $\cC$ is a
packing whose tangency map is isomorphic to $G$, we call $\cC$ a packing of
$G$.

\begin{thm}[Koebe-Andreev-Thurston Circle Packing Theorem \cite{K36,Th78}]
  Every finite simple planar map arises as the tangency map of a circle
  packing. If the map is a triangulation, the packing is unique up to
  M\"{o}bius transformations of the sphere.
\end{thm}

The \textbf{carrier} of a circle packing is the union of all the discs in
the packing together with the curved triangular regions enclosed between
each triplet of circles corresponding to a face (the \emph{interstices}).
Given some planar domain $D$, we say that a circle packing is \textbf{in
  $D$} if its carrier is $D$.

\begin{thm}[Rigidity for Infinite Packings, Schramm \cite{Schramm91}]
  \label{thm:rigidity}
  Let $G$ be a triangulation, circle packed in either $\C$ or $\D$. Then
  the packing is unique up to M\"{o}bius transformations preserving of $\C$
  or $\D$ respectively.
\end{thm}

It is often fruitful to think of packings in $\D$ as being circle packings
in (the Poincar\'{e} disc model of) the hyperbolic plane.  The uniqueness of
the packing in $\D$ up to M\"{o}bius transformations may then be stated
as uniqueness of the packing in the hyperbolic plane up to isometries of
the hyperbolic plane.

\medskip

The \textbf{vertex extremal length}, defined in \cite{HeSc}, from a vertex
to infinity on an infinite graph $G$ is defined to be
\begin{equation}
  \label{eq:VELdef}
  \VEL_G(v,\infty) = \sup_m \frac{\inf_{\gamma:v\to\infty}
    m(\gamma)^2}{\|m\|^2},
\end{equation}
where the supremum is over measures $m$ on $V(G)$ such that $\|m\|^2 = \sum
m(u)^2 < \infty$, and the infimum is over paths from $v$ to $\infty$ in $G$.
% (Clearly it suffices to consider simple paths.)
A connected graph is said to be \textbf{VEL parabolic} if $\VEL(v \to
\infty) = \infty$ for some vertex $v$ (and hence for any vertex) and
\textbf{VEL hyperbolic} otherwise.  The VEL type is monotone in the sense
that subgraphs of VEL parabolic graphs are also VEL parabolic.  A simple
random walk on any VEL hyperbolic graph is transient.  For graphs with
bounded degrees the converse also holds: Transient graphs with bounded
degrees are VEL hyperbolic~\cite{HeSc}.

\begin{thm}[He-Schramm \cite{HS93,HeSc}]
  Let $G$ be a one-ended, infinite, simple planar triangulation. Then $G$ may be
  circle packed in either the plane $\C$ or the unit disc $\D$, according
  to whether it is VEL parabolic or hyperbolic respectively.
\end{thm}

\medskip

The final classical fact about circle packing we will need is the following
quantitative version, due to Hansen \cite{Hansen}, of the Ring Lemma of
Rodin and Sullivan~\cite{RS87}, which will allow us to control the radii
along a random walk.

\begin{thm}[The Sharp Ring Lemma \cite{Hansen}]\label{lem:sharp_ring}
  Let $u$ and $v$ be two adjacent vertices in a circle packed triangulation, and
  $r(u),r(v)$ the radii of the corresponding circles.  There exists a universal
  positive constant $C$ such that
  \[
  \frac{r(v)}{r(u)} \leq e^{C\deg(v)}.
  \]
\end{thm}

\subsubsection{Measurability of Circle Packing}

At several points throughout the paper, we will want to define mass
transports in terms of circle packings. In order for these to be measurable
functions of the graph, we require measurability of the circle packing. Let
$(G,u,v)$ be a doubly rooted triangulation and let $\cC(G,u,v)$ be the
unique circle packing of $G$ in $\D$ or $\C$ such that the circle
corresponding to $u$ is centred at 0, the circle corresponding to $v$ is
centered on the positive real line and, in the parabolic case, the root circle has radius one.

Let $G_k$ be an exhaustion of $G$ by finite induced subgraphs with no
cut-vertices and such that the complements $G \setminus G_k$ are connected.
Such an exhaustion exists by the assumption that $G$ is a one-ended
triangulation.  Form a finite triangulation $G^*_k$ by adding an extra
vertex $\partial_k$ and an edge from $\partial_k$ to each boundary vertex
of $G_k$.

Consider first the case when $G$ is CP hyperbolic.
By applying a M\"obius transformation to some circle packing of $G^*_k$, we
find a unique circle packing $\cC^*_k$ of $G^*_k$ in $\C^\infty$ such
that the circle corresponding to $u$ is centred at the origin, the circle
corresponding to $v$ is centred on the positive real line and $\partial_k$
corresponds to the unit circle $\partial \D$.  In the course of the proof
of the He-Schramm Theorem, it is shown that this sequence of packings
converges to the unique packing of $G$ in $\D$, normalised so that the
circle corresponding to $u$ is centred at $0$ and the circle corresponding
to $v$ is centred on the positive real line.
%Such a packing is unique by rigidity (\cref{thm:rigidity}), so that there
%was in fact no need to take subsequences.

As a consequence, the centres and radii of the circles of $\cC(G,u,v)$ are
limits as $r\to\infty$ of the centre and radius of a graph determined by
the ball of radius $r$ around $u$.  In particular, they are pointwise
limits of continuous functions (with respect to the local topology on
graphs) and hence are measurable.

The hyperbolic radii are particularly nice to consider here. Since the
circle packing in $\D$ is unique up to isometries of the hyperbolic plane
(M\"obius maps), the hyperbolic radii do not depend on the choice of packing,
and we find that $r_h(v)$ is a function of $(G,v)$.

In the CP parabolic case, the same argument works except that the packing
$\cC^*_k$ of $G_k^*$ must be chosen to map $u$ to the unit circle and
$\partial_k$ to a larger circle also centred at $0$.

%%%%%%%%%%%%%%%%%%%%%%%%%%%%%%%%%%%%%%%%%%%%%%%%%%%%%%%%%%%%%%%%%%%%%%%%%%%%%%
\section{Characterisation of the CP type}\label{sec:type}

\begin{figure}
  \centering
  % the circle appears a bit smaller, but this is an illusion.
  \includegraphics[width=0.45\textwidth]{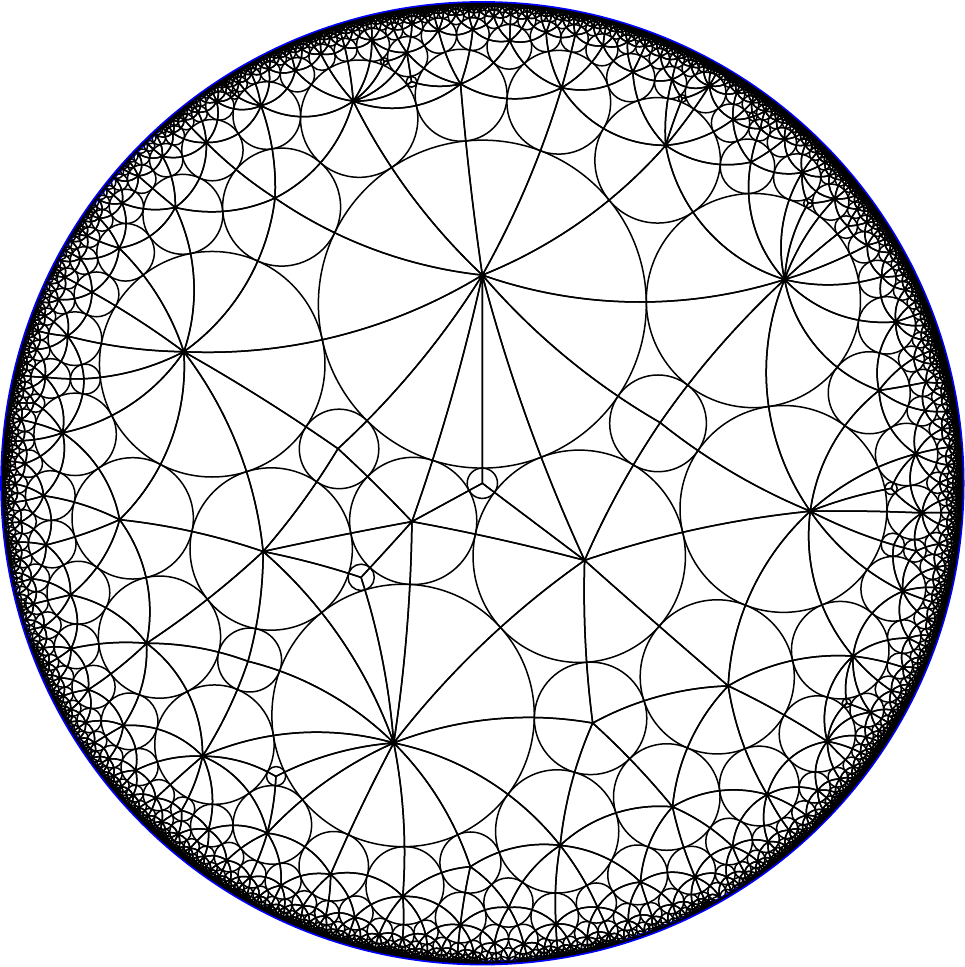}
  \qquad
  \includegraphics[trim = 0in 0.035in 0in 0.035in, clip, width=0.45\textwidth]{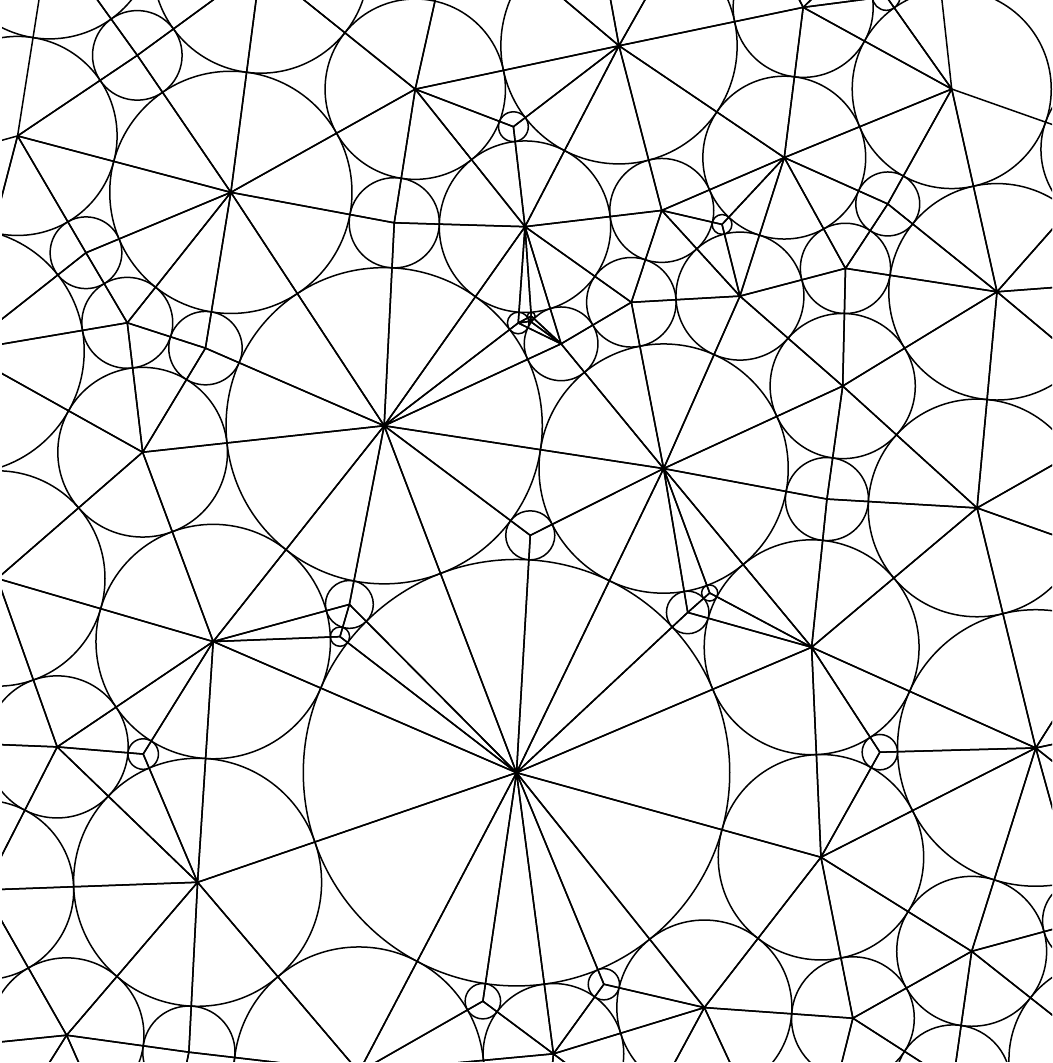}
  \caption{Circle packing induces an embedding of a triangulation with
    either hyperbolic or Euclidean geodesics, depending on CP type.  By
    rigidity (\cref{thm:rigidity}), the angles between pairs of adjacent
    edges do not depend on the choice of packing.}
  \label{fig:embedding}
\end{figure}

\begin{proof}[Proof of \cref{thm:classification}]
  Since $(G,\rho)$ is ergodic, and since the CP type does not depend on the choice of root, the CP type of $G$ is not random.  We first relate
  the circle packing type to the average degree.  Suppose $(G,\rho)$ is CP
  hyperbolic and consider a circle packing of $G$ in the unit disc.  Embed
  $G$ in $\D$ by drawing the hyperbolic geodesics between the hyperbolic
  centres of the circles in its packing, so that each triangle of $G$ is
  represented by a hyperbolic triangle (see \cref{fig:embedding}).  It is
  easy to see that this is a proper embedding of $G$.  By rigidity of the
  circle packing (\cref{thm:rigidity}), this drawing is determined by the
  isomorphism class of $G$, up to isometries of the hyperbolic plane.

  Define a mass transport as follows.  For each face $(u,v,w)$ of the
  triangulation with angle $\beta$ at $u$, transport $\beta$ from $u$ to
  each of $u,v,w$.  If $u$ and $v$ are adjacent, the transport from $u$ to $v$
  has contributions from both faces containing the edge, and the transport
  from $u$ to itself has a term for each face containing $u$.  By rigidity
  (\cref{thm:rigidity}), these angles are independent of the choice of
  circle packing, so that the mass sent from $u$ to $v$ is a measurable
  function of $(G,u,v)$.

  For each face $f$ of $G$, let $\theta(f)$ denote the sum of the internal
  angles in $f$ in the drawing.  The sum of the angles of a hyperbolic
  triangle is $\pi$ minus its area, so $\theta(f)<\pi$ for each face $f$.
  Each vertex $u$ sends each angle 3 times, for a total mass out of exactly
  $6\pi$.  A vertex receives mass
  \[
  \sum_{f :\,  u \in f} \theta(f) < \pi \deg(u).
  \]
  Applying the mass transport principle,
  \[
  6\pi < \pi\E[\deg(\rho)].
  \]
  Thus if $G$ is CP hyperbolic then $\E[\deg(\rho)]>6$.

  In the CP parabolic case, we may embed $G$ in $\C$ by drawing straight
  lines between the centres of the circles in its packing in the plane.  By
  rigidity, this embedding is determined up to translation and scaling, and
  in particular all angles are determined by $G$.  Since the sum of angles
  in a Euclidean triangle is $\pi$, the same transport as above applied in
  the CP parabolic case shows that $\E[\deg(\rho)]=6$.

  \medskip

  We now turn to amenability.  Euler's formula implies that the average
  degree of any finite simple planar graph is at most 6.  It follows that
  \[
  \alpha((G,\rho)) = \sup\left\{\E\left[\frac{\sum_{v \in
          K_\omega(\rho)}\deg_{\omega}(v)}{|K_\omega(\rho)|}\right] :
    \omega \text{ a finitary percolation}\right\} \leq 6.
  \]
  If $G$ is CP hyperbolic then $\E[\deg(\rho)]>6$, so that $\alpha((G,\rho))<\E[\deg(\rho)]$ and $(G,\rho)$ is invariantly non-amenable.
  % and hence $G$ has
  % positive invariant Cheeger constant $\mathbf{i}^{\mathrm{inv}}((G,\rho)) =
  % \E[\deg(\rho)]-\alpha((G,\rho)) > 0$. (This argument is valid even if $\E[\deg(\rho)]=\infty$.)

  Conversely, suppose $G$ is invariantly non-amenable.  By \cref{thm:BLS},
  $G$ almost surely admits a percolation $\omega$ which has positive
  Cheeger constant and bounded degrees.  Such an $\omega$ is transient and
  since it has bounded degree it is also VEL hyperbolic.  By monotonicty of
  the vertex extremal length, $G$ is almost surely VEL hyperbolic as well.
  The He-Schramm Theorem then implies that $G$ is almost surely CP
  hyperbolic.
\end{proof}

\begin{remark}
  In the hyperbolic case, let $\operatorname{Area}(u)$ be the total area of
  the triangles surrounding $u$ in its drawing. Since the angle sum in a
  hyperbolic triangle is $\pi$ minus its area, the mass transport that
  gives average degree greater than 6 in the hyperbolic case also
  gives\[\E[\deg(\rho)] = 6+\frac{1}{\pi}\E[\operatorname{Area}(\rho)],\]
  which relates the expected degree to the density of the circle packing.
\end{remark}

% \begin{remark}
% A related result of Beardon and Stephenson \cite{BeSt91} states that if $G$ is a one-ended triangulation in which the degree of every vertex
% \end{remark}
% We remark again that, since limits of simple planar graphs always have
% expected degree at most 6, we get that such limits are CP parabolic.
% This is the main step in the proof by Benjamini and Schramm that when
% such graphs have bounded degrees they are almost surely recurrent.  In
% \cite{BeSc} this step is done by a computation for the finite graphs,
% (applied for general finite sets in the plane) that gives a finitary,
% quantified version of being CP parabolic.  Our argument works directly in
% the infinite, unimodular setting.

%%%%%%%%%%%%%%%%%%%%%%%%%%%%%%%%%%%%%%%%%%%%%%%%%%%%%%%%%%%%%%%%%%%%%
\subsection{Completing the proof of the Benjamini-Schramm Theorem}
\label{sec:BSproof}

In this section we complete our new proof of the following theorem of
Benjamini and Schramm.

\begin{thm}[\cite{BeSc}]\label{thm:BeSc}
  Let $(G,\rho)$ be a weak local limit of finite planar graphs $G_n$ and
  suppose that $G$ has bounded degrees almost surely.  Then $(G,\rho)$ is
  almost surely recurrent.
\end{thm}

Recall that the number of \textbf{ends of a graph} $G$ is the supremum over
finite sets $K$ of the number of infinite connected components of $G
\setminus K$. As explained in \cite{BeSc}, it suffices to prove
\cref{thm:BeSc} when the graphs $G_n$ are simple triangulations. In this
case \cref{P:one_ended} implies a special case of the Benjamini-Schramm
Theorem: If $(G,\rho)$ is a simple one-ended triangulation that is a
Benjamini-Schramm limit of finite planar triangulations of uniformly
bounded degree, then $(G,\rho)$ is recurrent almost surely.  At the time,
this was the most difficult case.

Thus, to complete the proof of \cref{thm:BeSc} we need to consider the case
in which the limit $(G,\rho)$ has multiple ends.  We describe below two
different methods to handle this case.

\paragraph{Method 1.}

This proof considers separately three cases, depending on the number of
ends of $G$.  First, by combining Proposition 6.10 and Theorem 8.13 of
\cite{AL07}, we have the following.

\begin{prop}[\cite{AL07}]
  Let $(G,\rho)$ be an ergodic unimodular random rooted graph.  Then $G$
  has one, two or infinitely many ends almost surely.  If $(G,\rho)$ has
  infinitely many ends almost surely, it is invariantly non-amenable.
\end{prop}

We rule out the case of infinitely many ends by showing that local limits
of finite planar graphs are invariantly amenable.  Recall the celebrated
Lipton-Tarjan Planar Separator Theorem \cite[Theorem~2]{LT80} (which can
also be proved using circle packing theory \cite{MTTV97}).

\begin{thm}[\cite{LT80}]
  There exists a universal constant $C$ such that for every $m$ and every
  finite planar graph $G$, there exists a set $S\subset V(G)$ of size at
  most $Cm^{-1/2}|G|$ such that every connected component of $G \setminus
  S$ contains at most $m$ vertices.
\end{thm}

\begin{corollary}\label{cor:amen}
  Let $(G,\rho)$ be the local limit of a sequence of finite planar maps
  $G_n$ and suppose $\E[\deg(\rho)]<\infty$.  Then $(G,\rho)$ is invariantly
  amenable and hence has at most two ends.
\end{corollary}

\begin{proof}
  Let $\omega_n^m$ be a subset of $V(G_n)$ such that $G_n \setminus
  \omega_n^m$ has size at most $Cm^{-1/2}|G_n|$ and every connected
  component of $\omega_n^m$ has size at most $m$.  The sequence
  $(G_n,\rho_n,\omega_n^m)$ is tight and therefore has a subsequence
  converging to $(G,\rho,\omega^m)$ for some finitary percolation
  $\omega^m$ on $(G,\rho)$.  Since it is a limit of percolations on finite
  graphs with a uniform root, the limit is unimodular.

  We have that
  \[
  \P(\rho \in \omega^m) \geq 1-Cm^{-1/2} \xrightarrow[m\to\infty]{} 1.
  \]
  Similarly, $\P(X_1\in\omega^m)\to 1$.  By integrability of $\deg(\rho)$,
  we have that
  \[
  \E[\deg_{\omega^m}(\rho)] = \E[\mathbbm{1}(\rho, X_1 \in
  \omega^m)\deg(\rho)]\to \E[\deg(\rho)].
  \]
  Thus $\alpha((G,\rho))= \E[\deg(\rho)]$ and hence $(G,\rho)$ is invariantly
  amenable.
\end{proof}

Finally, we deal with the two-ended case.

\begin{prop}\label{prop:twoends}
  Let $(G,\rho)$ be a unimodular random rooted graph with exactly two ends
  almost surely and suppose $\E[\deg(\rho)]<\infty$.  Then $G$ is recurrent
  almost surely.
\end{prop}

\begin{proof}
  We prove the equivalent statement for $(G,\rho)$ reversible.  We may also
  assume that $(G,\rho)$ is ergodic.  Say that a finite set $S$
  \textbf{disconnects} $G$ if $G \setminus S$ has two infinite components.
  Since $G$ is two-ended almost surely, such a set $S$ exists and each
  infinite component of $G \setminus S$ is necessarily one-ended. We call
  these two components $G_1$ and $G_2$.  Suppose for contradiction that $G$
  is transient almost surely.  In this case, a simple random walk $X_n$
  eventually stays in one of the $G_i$, and hence the subgraph induced by
  this $G_i$ must be transient.

  Now, since $G$ is two-ended almost surely, there exist $R$ and $M$ such
  that, with positive probability, the ball $B_R(X_n)$ disconnects $G$ and
  $|B_R(X_n)| \leq M$.  By the Ergodic Theorem this
  occurs for infinitely many $n$ almost surely.  On the event that $X_n$
  eventually stays in $G_i$, since $G_i$ is one-ended, this yields an
  infinite collection of disjoint cutsets of size at most $M$ separating
  $\rho$ from infinity in $G_i$.  Thus, $G_i$ is recurrent by the
  Nash-Williams criterion \cite{LP:book}, a contradiction.
\end{proof}

\cref{thm:BeSc} now follows by combining
\cref{thm:classification,cor:amen,prop:twoends}.

\paragraph{Method 2.}

This proof reduces \cref{thm:BeSc} to \cref{thm:classification} by taking
universal covers.  Given a (not necessarily planar) map $M$, a
\textbf{cover} of $M$ is a map $\tilde M$ together with a surjective graph
homomorphism $\pi:\tilde M \to M$, such that for each vertex $v$, the
homomorphism $\pi$ maps the edges adjacent to $v$ bijectively to the edges
adjacent to $\pi(v)$ and preserves their cyclic ordering, and such that for
each face $f$, $\pi$ maps the edges adjacent to $f$ bijectively to the
edges adjacent to $\pi(f)$.  The \textbf{universal cover} of $M$ is a
cover $\pi : \tilde M \to M$ such that $\tilde M$ is
simply connected.  If $M$ is drawn on a surface $S$, the universal cover
$\tilde M$ of $M$ may be constructed by taking every lift of every edge of
$M$ in $S$ to the universal cover $\tilde S$ of $S$ (see e.g.\ \cite{Hatch}
for the topological notions of universal cover and path lifting).
Alternatively, the universal cover $\tilde M$ may be constructed directly
as in \cite{St05}.  The universal cover is unique in the sense that if
$\pi' : \tilde M' \to M$ is also a universal cover of $M$ then there exists
an isomorphism of maps $f:\tilde M' \to \tilde M$ such that $\pi' = \pi
\circ f$.  Note that if a cover $\pi:\tilde M\to M$ is a cover of a map $M$
and $\tilde M$ is recurrent, the projection
$X_n=\pi(\tilde X_n)$ of a simple random walk $\tilde X_n$ on $\tilde M$ is
a simple random walk on $M$, and it follows that $M$ is also recurrent.

Let $(M,\rho)$ be a unimodular random rooted map with universal cover $\pi
\! : \tilde M \to M$. Let $\tilde \rho$ be chosen arbitrarily from the
preimage $\pi^{-1}(\rho)$; The isomorphism class of the rooted map $(\tilde
M, \tilde \rho)$ does not depend on this choice.  We claim that the random
rooted map $(\tilde M,\tilde \rho)$ is unimodular.
To see this, recall that a random rooted graph is unimodular if and only if it is \emph{involution invariant} \cite[Proposition 2.2]{AL07}, meaning that
\[\E\sum_{v\in V} f(G,\rho,v) = \E\sum_{v\in V}f(G,v,\rho)\]
% the mass transport principle holds for all mass-transports
 % that it suffices to assume the mass
% transport principle for functions
whenever
$f$ is a mass-transport such that $f(G,u,v)$ is zero unless
$u$ and $v$ are adjacent in $G$. The equivalence of unimodularity and involution invariance extends immediately to random rooted maps.
Given such an $f:\mathcal{M}_{\bullet\bullet}\to[0,\infty]$, let $g:\mathcal{M}_{\bullet\bullet}\to[0,\infty]$ be defined to be
\[g(M,u,v) = \sum_{e:\, e^-=u,\,e^+=v}f(\tilde M, \tilde e^-, \tilde e^+),\]
where $\pi:\tilde M \to M$ is the universal cover of $M$ and $\tilde e$ is an arbitrary element of $\pi^{-1}(e)$ for each oriented edge $e$ of $M$ (by uniqueness of the universal cover, the value of $g$ does not depend on this choice).
 % Let $\Gamma$ be the group of deck transformations of $\tilde M$, that is, map homomorphisms $\gamma$ from $\tilde M$ to itself such that
Then $g$ is a mass transport and, letting $\tilde V$ denote the vertex set of $\tilde M$, we have
\[ \sum_{v \in V}g(M,\rho,v) = \sum_{\tilde v\in \tilde V}f(\tilde M,\tilde \rho, \tilde v)
% and
 \quad \text{ and } \quad \sum_{v \in V}g(M,v,\rho) = \sum_{\tilde v\in \tilde V}f(\tilde M, \tilde v,\tilde \rho),\]
% \begin{align*}
% \sum_{\tilde v\in \tilde V}f(\tilde M, \tilde v,\tilde \rho) = \sum_{\tilde v \in \pi^{-1}(v)}\sum_{v\in V}f(\tilde M, \tilde v,\tilde \rho)\\
% \sum_{v \in V}g(M,v,\rho),\end{align*}
so that we deduce involution invariance of $(\tilde M,\tilde \rho)$ from involution invariance of $(M,\rho)$.
  % Such a function on $\tilde M$ projects to a
% function on $M$, and involution invariance of $(\tilde M,\tilde \rho)$ follows from involution invariance of $(M,\rho)$.
% This is easiest to see in the case that $\E[\deg(\rho)]<\infty$, where
% the reversibility of the degree-biasing of $(\tilde M, \tilde \rho)$
% follows immediately from the reversibility of the degree-biasing of $(M,
% \rho)$ and the uniqueness of the universal cover.  Without this
% assumption, it suffices to check \emph{involution invariance} of $(\tilde
% M,\tilde \rho)$, see \cite[Proposition 2.2]{AL07}.
Furthermore, if $(M,\rho)$ is ergodic then $(\tilde M,\tilde \rho)$ is also
ergodic.  Indeed, for every invariance event $A \subseteq
\mathcal{M}_\bullet$, the event $\{(G,\rho) \in \mathcal{M}_\bullet :
(\tilde G, \tilde \rho) \in \mathcal{M}_\bullet\}$ is also invariant to
changing the root, and it follows that if $(M,\rho)$ is ergodic then
$(\tilde M,\tilde \rho)$ is also ergodic.

\begin{proof}[Alternative proof of \cref{thm:BeSc}]
  Let $(G,\rho)$ be a simple, bounded degree, ergodic unimodular random
  rooted triangulation with $\E[\deg(\rho)]=6$.  The universal cover
  $(\tilde G, \tilde \rho)$ of $(G,\rho)$ has all these properties and is
  also one-ended, so that $\tilde G$ is CP parabolic almost surely by
  \cref{thm:classification}.  By the He-Schramm Theorem \cite{HeSc},
  $\tilde G$ is recurrent almost surely, and so $G$ is also recurrent
  almost surely, completing the proof of the Benjamini-Schramm Theorem.
\end{proof}

%%%%%%%%%%%%%%%%%%%%%%%%%%%%%%%%%%%%%%%%%%%%%%%%%%%%%%%%%%%%%%%%%%
\section{Boundary Theory}\label{sec:boundarytheory}

Recall that given a $G$ and a vertex $v$ we write $P^G_v$ and $E^G_v$
to denote the probability and expectation with respect to random walk
$(X_n)_{n \geq 0}$ on $G$ started from $v$.

\subsection{Convergence to the boundary}

Let $(G,\rho)$ be a one-ended, simple, CP hyperbolic reversible random
triangulation. Recall that for a CP hyperbolic $G$ with circle packing
$\cC$ in $\D$, we write $r(v)$ and $z(v)$ for the Euclidean radius and
centre of the circle corresponding to the vertex $v$ in $\cC$ and
$z_h(v),r_h(v)$ for the hyperbolic centre and radius.

Our first goal is to show that the Euclidean radii $r(X_n)$ decay
exponentially along a random walk $(X_n)$.  We initially prove only a
bound, and will prove the existence of the limit rate of decay stated in
\cref{thm:exp_decay} only after we have proven the exit measure is
non-atomic.

\begin{lemma}
  \label{lem:exp_decay}
  Let $(G,\rho)$ be a CP hyperbolic reversible random rooted triangulation
  with $\E[\deg(\rho)]<\infty$ and let $\mathcal{C}$ be a circle packing of
  $G$ in the unit disc.  Let $(X_n)_{n \ge 0}$ be a simple random walk on
  $G$ started from $\rho$.  Then almost surely
  \[
  \limsup_{n \to \infty} \frac{\log r(X_n)}{n} < 0.
  \]
\end{lemma}

\begin{proof}
  We may assume that $(G,\rho)$ is ergodic, else we may take an ergodic
  decomposition.  By \cref{thm:classification} $(G,\rho)$ is invariantly
  non-amenable.  By \cref{thm:BLS}, there is an ergodic percolation
  $\omega$ on $G$ such that $\deg(v)$ is bounded by some $M$
  for all $v\in \omega$ and $\mathbf{i}_E(\omega)>0$ almost surely.

  Recall the notion of an \textbf{induced random walk} on $\omega$: let
  $N_m$ be the $m$th time $X$ is in $\omega$ (that is, $N_0 = \inf\{n \geq
  0 : X_n \in \omega\}$ and inductively $N_{m+1}=\inf \{n > N_m : X_n \in
  \omega\}$).  The \textbf{induced network} $\bar\omega$ is defined to be the
  weighted graph on the vertices of $\omega$ with edge weights given by
  \[
  \bar{w}(u,v) = \deg(u) P^G_u(X_{N_1}=v)
  \]
  so that $X_{N_m}$ is the random walk on the weighted graph $\bar \omega$.
  Note that $\bar\omega$ may have non-zero weights between vertices which
  are not adjacent in $G$, so that $\bar\omega$ is no longer a percolation
  on $G$, and may not even be planar.

  We first claim that $\bar\omega$ with the weights $\bar{w}$ of the
  induced random walk also has positive Cheeger constant.  Indeed, the
  weight of a vertex $v \in \bar\omega$ is just its degree in $G$
  and so is between $1$ and $M$ for any vertex.  The edge boundary
  in $\bar\omega$ of a set $K$ is at least the number of edges connecting
  $K$ to $V\setminus K$ in $\omega$.  Thus $\mathbf{i}_E(\bar \omega) \geq
  \mathbf{i}_E(\omega)/M > 0$.  It follows that the induced
  random walk on $\omega$ has spectral radius less than one
  \cite{LP:book}.

  Now, as in \eqref{eq:CS}, Cauchy-Schwarz gives that, for some $c>0$,
  \[
    P^G_\rho(X_{N_m} = v) \leq  M^{1/2}\exp(-cm)
  \]
  for every vertex $v$ almost surely.  Since the total area of all circles
  in the packing is at most $\pi$,  with $c$ as above, there
  exists at most $e^{cm/2}$ circles of radius greater than
  $e^{-cm/4}$ for each $m$. Hence
  \begin{align*}
    P^G_\rho\left( r(X_{N_m}) \geq e^{-cm/4} \right)
    &= \sum_{v : r(v) \geq e^{-cm/4}} P^G_\rho\left(X_{N_m} = v\right) \\
    &\le \left|\left\{ v : r(v) \geq e^{-cm/4} \right\}\right| \cdot
    M^{1/2} e^{-cm} \\
    &\le M^{1/2} e^{-cm/2}.
  \end{align*}
  These probabilities are summable, and so Borel-Cantelli implies that
  almost surely for large enough $m$,
  \[
    r(X_{N_m}) \leq e^{-cm/4}.
  \]
  That is, we have exponential decay of the radii for the induced walk:
  \begin{equation}
    \limsup_{m \to \infty} \frac{\log r(X_{N_m})}{m} \leq -\frac{c}{4}.
  \end{equation}

  It remains to prove that the exponential decay is maintained between
  visits to $\omega$.  By stationarity and ergodicity of $(G,\rho,\omega)$,
  the density of visits to $\omega$ is $\P(\rho\in\omega) \neq 0$.  That is,
  \[
  \lim_{m \to \infty} \frac{N_m}{m} = \P(\rho\in\omega)^{-1}
  \]
  almost surely.  In particular,
  \begin{equation}
    \label{eq:exp_decay1}
    \limsup_{m \to \infty} \frac{\log r(X_{N_m})}{N_m} < 0.
  \end{equation}

  Given $n$, let $m$ be the number of visits to $\omega$ up to time $n$, so
  that $N_m\leq n < N_{m+1}$.  Since $m/N_m$ converges, $n/N_m\to 1$.
  By the Sharp Ring Lemma (\cref{lem:sharp_ring}),
  \[
  \frac{r(X_n)}{r(X_{N_m})} \le \exp \left(C\sum_{i=N_m}^n{\deg(X_i)}\right)
  \]
  so that
  \begin{equation}
    \frac{\log r(X_n)}{n} \leq \frac{\log r(X_{N_m})}{n}
    + \frac{C}{n} \sum_{i=N_m}^{n}\deg(X_i). \label{eq:exp_decay2}
  \end{equation}
  Again, the Ergodic Theorem gives us the almost sure limit
  \[
  \lim \frac{1}{n} \sum^n_{i=0}\deg(X_i) = \E[\deg(\rho)].
  \]
  Thus almost surely
  \begin{align*}
    \lim \frac{1}{n} \sum_{i=N_m}^{n}\deg(X_i) &=
    \lim\frac{1}{n}\sum_{i=0}^{n}\deg(X_i)- \lim
    \frac{N_m}{n}\lim\frac{1}{N_m}\sum_{i=0}^{N_m}\deg(X_i) \\
    &=\E[\deg(\rho)]-\E[\deg(\rho)]=0.
  \end{align*}
  Combined with \eqref{eq:exp_decay1} and \eqref{eq:exp_decay2} we get
  \begin{align*}
  \limsup_{{n\to \infty}} \frac{\log r(X_n)}{n}
  &= \limsup_{{n\to \infty}} \frac{\log r(X_{N_m})}{n} \\
  &= \lim_{{n\to \infty}} \frac{N_m}{n}
  \limsup_{{n\to \infty}} \frac{\log r(X_{N_m})}{N_m} < 0.  \qedhere
  \end{align*}
\end{proof}

\begin{proof}[Proof of \cref{thm:boundary}, \cref{thm:convergence}]
  We prove the equivalent statement for $(G,\rho)$ reversible with
  $\E[\deg(\rho)]<\infty$, putting us in the setting of
  \cref{lem:exp_decay}.  The path formed by drawing straight lines between
  the Euclidean centres of the circles along the random walk path has
  length $r(\rho)+2\sum_{i\geq1} r(X_i)$, which is almost surely finite by
  \cref{lem:exp_decay}. It follows that the sequence of Euclidean centres
  is Cauchy almost surely and hence converges to some point, necessarily in
  the boundary.  Because the radii of the circles $r(X_n)$ converge to zero
  almost surely, the hyperbolic centres must also converge to the same
  point.
\end{proof}

%%%%%%%%%%%%%%%%%%%%%%%%%%%%%%%%%%%%%%%%%%%%%%%%%%%%%%%%%%%%%%%%%
\subsection{Full support and non-atomicity of the exit measure}
\label{sec:exitmeasure}

We now prove item $2$ of \cref{thm:boundary}, which states that the exit
measure on the unit circle has full support and no atoms almost surely.  We
start with a general observation regarding atoms in boundaries of
stationary graphs.

We say that two metrics $d_1$ and $d_2$ on the vertex set $V$ of a graph
$G$ are \textbf{compatible} if the identity map from $V$ to itself extends
to an isomorphism between the completions of the metric spaces $(V,d_1)$
and $(V,d_2)$ (or, equivalently, if the same sequences are Cauchy for $d_1$
and $d_2$).  For example, the Euclidean distances between centres of
circles corresponding to vertices in different circle packings of a
one-ended planar triangulation in either the full plane or the unit disc
are compatible by \cref{thm:rigidity}.
% Let $A\subseteq \mathcal{G}_\bullet$ be rerooting-invariant and suppose that $d=d^G_\rho(u,v)$ is a positive measurable function defined on the space of triply rooted graphs $(G,\rho,u,v)$ (i.e., graphs with an ordered triple of distinguished vertices with an appropriate variant of the local topology), and suppose that
% for every graph $G$ in $A$ (i.e., every locally finite connected graph $G$ such that $(G,u)\in A$ for every vertex $u$ of $G$), we have
We define a \textbf{compatible family of metrics} to be a Borel function
$d=d^G_u(v,w)$ from the space of triply rooted graphs (i.e., the set of
isomoprhism classes of graphs with an ordered triple of distinguished
vertices, equipped with an appropriate variant of the local topology) to
the positive reals such that for every locally finite, connected graph
$G=(V,E)$,
\begin{enumerate}
\item $d^G_u(\cdot,\cdot)$ is a metric on $V$ for every vertex $u$ of $G$,
  and
\item the metrics $d^G_u$ and $d^G_v$ are compatible for each two vertices
  $u$ and $v$ of $G$.
\end{enumerate}

Given a compatible family of metrics $d$ and a rooted graph $(G,\rho)$, the
completion $\bar{V}$ of $V$ with respect to $d^G_\rho$ has a topology that
does not depend on the choice of root vertex $\rho$.  Such a completion is
called an \textbf{invariant completion}.  Compatible families
of metrics for maps are defined similarly.  In some cases of interest, a
compatible family of metrics $d$ might only be defined for graphs or maps
in some rerooting-invariant class (e.g.\ circle packings are defined for
the class of one-ended simple planar triangulations).  In this case, we may
extend $d$ arbitrarily to all graphs or maps by setting it to be the
discrete metric where it is not defined.

\begin{lemma}\label{lem:atom}
  Let $d$ be a compatible family of metrics, let $(G,\rho)$ be a stationary
  random rooted graph or map, and let $\bar V$ be the completion of $V$
  with respect to $d^G_\rho$.  Suppose that the random walk on $G$ converges
  almost surely to a point in the boundary $\partial V = \bar{V}\setminus
  V$.  Then the exit measure on $\partial V$ is either trivial (concentrated
  on a single point) or non-atomic almost surely.
\end{lemma}

For each CP hyperbolic one-ended simple planar triangulation $G$, take a
circle packing of $G$ in $\D$, normalized so that the circle corresponding
to $u$ is centred at the origin, and let $d=d^G_u(v,w)$ be the Euclidean
distance between the Euclidean centres of the circles corresponding to $v$
and $w$.  By circle packing rigidity (\cref{thm:rigidity}), this circle
packing is unique up to rotations, so that the metric $d^G_u$ is well
defined.  The metrics $d^G_u$ and $d^G_v$ are also compatible for every
pair of vertices $u$ and $v$ in $G$.  Thus, after an arbitrary extension to
other maps, $d$ is an compatible family of metrics.

Another natural example, defined for all graphs, is the Martin
compactification.  As a consequence of this lemma, for any stationary
random graph, the exit measure on the Martin boundary (which can be defined
as the completion of $V$ with respect to a compatible family of metrics,
see e.g.\ \cite{Woess}) is almost surely either non-atomic or trivial.  In
particular, this gives an alternative proof of a recent result of
Benjamini, Paquette and Pfeffer \cite{BPP14}, which states that for every
stationary random graph, the space of bounded harmonic functions on the
graph is either one dimensional or infinite dimensional almost surely.  A
straightforward extension of this lemma applies to random families of
metrics.

\begin{proof}
  Condition on $(G,\rho)$. For each atom $\xi$ of the exit measure, define the
  harmonic function $h_{\xi} (v) = P^G_v( \lim X_n = \xi)$. By L\'{e}vy's
  0-1 law,
  \[
  h_{\xi}(X_n)\xrightarrow[]{a.s.}
    \mathbbm{1}( \lim X_n = \xi)
  \]
  for each atom $\xi$ and
  \[
  P^G_{X_n}(\lim X_n \text{ is an atom}) = \sum_\xi h_\xi(X_n)
  \xrightarrow[]{a.s.} \mathbbm{1}(\lim X_n \text{ is an atom}).
  \]
  Define $M(G,v) = \max_{\xi} h_{\xi}(v)$ to be the maximal atom size. Since the topology of $\bar V$ does not depend on the choice of root, the sequence $M(G,X_n)$ is stationary.
  Combining the two above limits, we find the almost sure limit
  \[
  M(G,X_n) \xrightarrow[]{a.s.} \mathbbm{1}(\lim X_n \text{ is an atom}).
  \]
 Since $M(G,X_n)$ is a stationary sequence with limit in $\{0,1\}$, it follows that
  $M(G,\rho) \in \{0,1\}$ almost surely.  That is,
  either there are no atoms in the exit measure or there is a single atom
  with weight 1 almost surely.
\end{proof}

\begin{proof}[Proof of \cref{thm:boundary}, item 2.]
  We may assume that $(G,\rho)$ is ergodic.  Applying \cref{lem:atom} to
  the $\deg(\rho)$-biasing of $(G,\rho)$, we deduce that the exit measure
  has at most a single atom.
  Next, we rule out having a single atom.
   Suppose for contradiction
  that there is a single atom $\xi = \xi(\cC)$ almost surely for some (and
  hence every) circle packing $\cC$ of $G$ in $\D$.  Applying the M\"obius
  transformation
  \[
  \Phi(z) = -i\frac{z+\xi}{z-\xi},
  \]
  which maps $\D$ to the upper half-plane $\H=\{\Im(z)>0\}$ and $\xi$ to
  $\infty$, gives a circle packing of $G$ in $\H$ such that the random walk
  tends to $\infty$ almost surely.  Since circle packings in $\H$ are
  unique up to M\"obius transformations and the boundary point $\infty$ is
  determined by the graph $G$, such a circle packing in $\H$ is unique up
  to M\"obius transformations of the upper half-plane that fix $\infty$,
  namely $az+b$ with real $a\geq 0$ and $b$ (translations and dilations).

  Inverting around the atom has therefore given us a way of canonically
  endowing $G$ with Euclidean geometry: if we draw $G$ in $\H$ using straight
  lines between the Euclidean centres of the circles in the half-plane
  packing, the angles at the corners around each vertex $u$ are independent
  of the original choice of packing $\mathcal{C}$.
  Transporting each angle from $u$ to each of the three vertices forming
  the corresponding face $f$ as in the proof of \cref{thm:classification}
  implies that $\E[\deg(\rho)]=6$.  This contradicts \cref{thm:classification}
  and the assumption that $(G,\rho)$ is CP hyperbolic almost surely.  This
  completes the proof that the exit measure is non-atomic.

  \medskip

  To finish, we show that the exit measure has support $\partial\D$.
  Suppose not.  We will define a mass transport on $G$ in which each vertex
  sends a mass of at most one but some vertices receive infinite mass,
  contradicting the mass transport principle.

  Consider the complement of the support of the exit measure, which is a
  union of disjoint open intervals $\bigcup_{i \in I}(\theta_i,\psi_i)$ in
  $\partial \D$.  Since the exit measure is non-atomic, $\theta_i \neq
  \psi_i$ mod $2\pi$ for all $i$.

  \begin{figure}
    \centering
    \begin{subfigure}[t]{0.45\textwidth}
      \includegraphics[width=\textwidth]{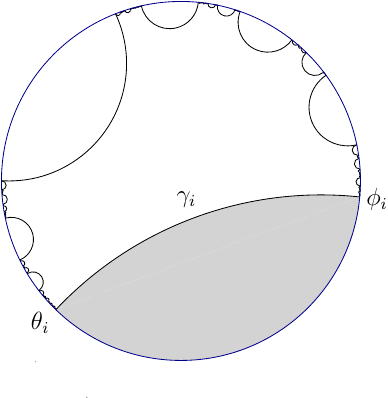}
      \caption{A geodesic $\gamma_i$ is drawn over each component of the
        complement of the support. Circles contained in the shaded area are
        in $A_i$.}
    \end{subfigure}
    \quad
    \begin{subfigure}[t]{0.45\textwidth}
      \includegraphics[width=\textwidth]{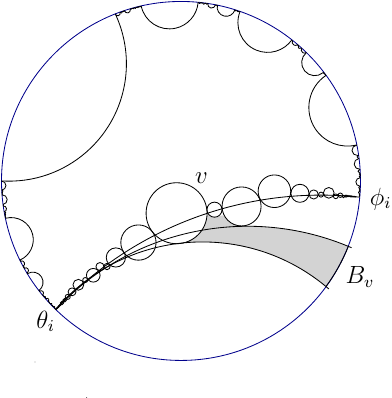}
      \caption{The vertex $v$ receives mass from circles with
        hyperbolic centres in the shaded area.}
    \end{subfigure}
    \caption{An illustration of the mass transport used to show the exit
      measure has full support.}
    \label{fig:support}
  \end{figure}

  For each such interval $(\theta_i,\psi_i)$, let $\gamma_i$ be the
  hyperbolic geodesic from $e^{i\theta_i}$ to $e^{i\psi_i}$.  That is,
  $\gamma_i$ is the intersection with $\D$ of the circle passing through
  both $e^{i\theta_i}$ and $e^{i\psi_i}$ that intersects $\partial \D$ at
  right angles.  Let $A_i$ be the set of vertices such that the circle
  corresponding to $v$ is contained in the region to the right of
  $\gamma_i$, i.e.\ bounded between $\gamma_i$ and the boundary interval
  $(\theta_i,\psi_i)$ (see \cref{fig:support}(a)).

  Each vertex is contained in at most one such $A_i$.  For each vertex $u$
  in $A_i$, consider the hyperbolic geodesic ray $\gamma_u$ from the
  hyperbolic centre $z_h(u)$ to $e^{i\theta_i}$.  Define a mass transport
  by sending mass one from $u\in A_i$ to the vertex $v$ corresponding to
  the first circle intersected by both $\gamma_u$ and $\gamma_i$.  There
  may be no such circle, in which case no mass is sent from $u$.  Since the
  transport is defined in terms of the hyperbolic geometry and the support
  of the exit measure, it is a function of the isomorphism class of
  $(G,u,v)$ by \cref{thm:rigidity}.

  Let $\phi \in (\theta_i, \psi_i)$ and consider the set of vertices whose
  corresponding circles intersect both $\gamma_i$ and the geodesic
  $\gamma_\phi$ from $e^{i\phi}$ to $e^{i\theta_i}$. As $\phi$ increases
  from $\theta_i$ to $\psi_i$, this set is increasing.  It follows that for
  each fixed $v$ for which the circle corresponding to $v$ intersects
  $\gamma_i$, the set $B_v$ of $\phi \in (\theta_i, \psi_i)$ for which the
  circle corresponding to $v$ is the first circle intersected by
  $\gamma_\phi$ that also intersects $\gamma_i$ is an interval (see
  \cref{fig:support}(b)).

  Since there are only countably many vertices, $B_v$ must have positive
  length for some $v$.  Thus there is an open neighbourhood of the boundary
  in which all the circles send mass to this vertex.  This vertex therefore
  receives infinite mass, contradicting the mass transport principle.
\end{proof}

%%%%%%%%%%%%%%%%%%%%%%%%%%%%%%%%%%%%%%%%%%%%%%%%%%%%%%%%%%%%%%%%%%%%%%%%%%%%%%%
\subsection{The unit circle is the Poisson boundary}
\label{subsec:poisson}

The Poisson-Furstenberg boundary \cite{Furst} (or simply the Poisson
boundary) of a graph (or more generally, of a Markov chain) is a formal way
to encode the asymptotic behaviour of random walks on $G$.  We refer the
reader to \cite{Kai,Pete,LP:book} for more detailed introductions.

Recall that a function $h: V(G) \to \R$ is said to be \textbf{harmonic} if
\[
h(v)=\frac{1}{\deg(v)} \sum_{u \sim v} h(u)
\]
for all $v \in V(G)$ --- or, equivalently, if $h(X_n)$ is a martingale.
Let $\bar{G} = G \cup \partial G$ be a compactification of $G$ so that the
random walk $X_n$ converges almost surely. For each $v\in V(G)$ we let
$P^G_v$ denote the law of the limit of the random walk started at $v$. Every bounded Borel function $g$ on $\partial G$ extends to a harmonic
function \[h(v) := E^G_v \big[g(\lim X_n)\big]\] on $G$.
Such a compactification is called a \textbf{realisation of the Poisson
  boundary} of $G$ if \emph{every} bounded harmonic funtion $h$ on $G$ may be
represented as an extension of a  boundary function in this way.

Harmonic functions can be used to encode asymptotic behaviour of the random
walk as follows.
Let $G^\mathbb{N}$ be the space of sequences in $G$.
  The shift operator on $G^\mathbb{N}$ is defined by
$\theta(x_0,x_1,\dots) = (x_1,x_2,\dots)$, and  we write $\cI$ for the
$\sigma$-algebra of shift-invariant events $A = \theta A$.
Be careful to note the distinction between invariant events for the random
walk on $G$, just defined, and invariant events for the sequence
$(G,(X_{n+k})_{n \in \Z})_{k \in \Z}$ as defined in
\cref{subsec:reversibility}.

There is an isomorphism between the space of bounded harmonic functions on
$G$ and $L^\infty(G^\N,\cI)$ given by
\begin{align*}
  h \mapsto g(x_1,x_2,\dots) &= \lim_{n \to \infty} h(x_n),  &
  g \mapsto h(v) &= E^{G}_v\big[g(v,X_1,X_2,\dots)\big].
\end{align*}
The limit here exists $P^G_v$-almost surely by the bounded Martingale
Convergence Theorem, while the fact that these two mappings are inverses of
one another is a consequence of L\'evy's 0-1 Law: If $h(v) = E^G_v[g(X)]$
is the harmonic extension of some invariant function $g$, then
\begin{equation}
  \label{eq:levy}
  h(X_n) \xrightarrow{a.s.} g(\rho,X_1,X_2,\dots).
\end{equation}

As a consequence of this isomorphism, and since the span of simple
functions is dense in $L^\infty$, the topological boundary $\partial G$ is a realisation of the
Poisson boundary of $G$ if and only if for every invariant event $A$ there
exists a Borel set $B\subset\partial G$ such that the symmetric difference
$A\Delta\{\lim X_n\in B\}$ is $P^G_v$-null.

For example, the Poisson boundary of a tree may be realised as its space of
ends, and the one-point compactification of a transient graph G gives rise
to a realisation of the Poisson boundary if and only if $G$ is Liouville
(i.e.\ the only bounded harmonic functions on $G$ are constant). The
Poisson boundary of any graph may be realised as the graph's
Martin boundary \cite{Woess}, but this is not always the most natural
construction.

\medskip

Our main tools for controlling harmonic functions will be L\'evy's 0-1 Law
and the following consequence of the Optional Stopping Theorem.  For a set
$W\subset V$ of vertices, let $T_W$ be the first time the random walk
visits $W$.  If $h$ is a positive, bounded harmonic function, the Optional
Stopping Theorem implies
\begin{equation}\label{eq:ostcor}
  h(v) \geq E^G_v[h(X_{T_W})\mathbbm{1}_{T_W<\infty}]
  \geq P^G_v(\text{Hit } W) \inf\{h(u) : u \in W\}.
\end{equation}

\begin{lem}\label{lem:Poissonlem}
  Let $(G,\rho)$ be a CP hyperbolic reversible random rooted triangulation
  with $\E[\deg(\rho)] < \infty$, and let $(X_n)_{n\in\Z}$ be the reversible
  bi-infinite random walk.  Then almost surely
  \[
  P^G_{X_n}\left(\text{hit } \{X_{-1},X_{-2},\dots\}\right) \to 0.
  \]
\end{lem}

\begin{proof}
  Let $\mathcal{C}$ be a circle packing of $G$ in $\D$.  Recall from
  \cref{thm:boundary}, \cref{thm:convergence} that for a random walk $X_n$, almost surely $\Xi
  := \lim z(X_n)$ exists, and its law is non-atomic and of full support on
  $\partial\D$.  Since the exit measure is non-atomic, the limit points
  $\Xi_+ := \lim z(X_n)$ and $\Xi_- := \lim z(X_{-n})$ are almost surely
  distinct.

  Let $\{U_i\}_{i \in I}$ be a countable basis for the topology of
  $\partial\D$ (say, intervals with rational endpoints) and for each $i$
  let $h_i$ be the harmonic function
  \[
  h_i(v) = P^G_v\left(\Xi \in U_i\right).
  \]
  By L\'evy's 0-1 law, $h_i(X_n) \to \mathbbm{1}\!\left(\Xi_+\in
    U_i\right)$ for every $i$ almost surely.  Thus there
  exists some $i_0$ with $\Xi_- \in U_{i_0}$ and $\Xi_+
  \notin U_{i_0}$.  In particular there is almost surely some bounded
  harmonic function $h=h_{i_0}\geq 0$ with $h(X_n)
  \xrightarrow[n\to\infty]{} 0$ and
  \[
  a := \inf\{h(X_{-m}) : m>0\} > 0.
  \]
  By \eqref{eq:ostcor}
  \[
    h(X_n) \geq a \cdot P^G_{X_n}\left(\text{hit }
      \{X_{-1},X_{-2},\dots\}\right).
  \]
  Since $h(X_n)\to 0$, we almost surely have
  \[
    P^G_{X_n}\left(\text{Hit } \{X_{-1},X_{-2},\dots\}\right) \to 0.
    \qedhere
  \]
\end{proof}

\begin{proof}[Proof of \cref{thm:boundary}, \cref{thm:Poisson}]
  We prove the equivalent statement for $(G,\rho)$ reversible with
  $\E[\deg(\rho)]<\infty$, and may assume that $(G,\rho)$ is ergodic.

  We need to prove that for every invariant event $A$ for the simple
  random walk on $G$ with $P^G_\rho(A)>0$, there is a Borel set $B
  \subset \partial \D$ such that
  \[
%  P^G_\rho\left(A\, \Delta \left\{\Xi_+ = \lim_{n \to \infty} z(X_n) \in B
%    \right\} \right) = 0.
  P^G_\rho\left(A\, \Delta \left\{\Xi_+ \in B
    \right\} \right) = 0,
  \]
  where $\Xi_+=\lim z(X_n)$.
  Let $h$ be the harmonic function $h(v) = P^{G}_v(A)$, and let $B$ be the
  set of $\xi \in \partial \D$ such that there exists a path
  $(\rho,v_1,v_2,\dots)$ in $G$ such that for some $c>0$,
  \[
  h(v_i) \to 1,
  \quad
  z(v_i) \to \xi,
  \quad \text{ and }\quad
  |\xi -z(v_i)| < 2e^{-ci},
  \]
  where $|\cdot|$ denotes Euclidean length.  The condition
  on exponential decay of $|\xi - z(v_i)|$ can be omitted by invoking the
  theory of universally measurable sets. We are spared from
  this by \cref{lem:exp_decay}.
  With an explicit rate of convergence, it is straightforward to see that
  $B$ is Borel: Let $B_{c,m,\eps,n}$ be the open set of $\xi\in\partial\D$
  such that there exists a path $\rho,v_1,\dots,v_n$ in $G$ such that $h(v_i)>1-\eps$
  for every $i\geq m$, and with $|\xi-z(v_i)| < 2e^{-ci}$. Then $B =
  \bigcup_c \bigcap_\eps \bigcup_m \bigcap_n B_{c,m,\eps,n}$, where $m,n$
  are integers and $c,\eps$ are positive rationals, and it follows that $B$ is Borel.

  If the random walk has $(\rho,X_1,\dots) \in A$ then, by L\'{e}vy's 0-1
  law and \cref{lem:exp_decay}, the limit point $\Xi_+$ is in $B$
  almost surely.  In particular, if $P^G_\rho(A)>0$, then the exit measure
  of $B$ is positive.  It remains to show that $(\rho,X_1,\dots) \in A$
  almost surely on the event that $\Xi_+ \in B$.

  Consider the two intervals $L$ and $R$ separating the almost surely
  distinct limit points $\Xi_+$ and $\Xi_-$.  Let $p^n_L$
  and $p^n_R$ be the probabilities that a \emph{new, independent} random
  walk started from $X_n$ hits the boundary in the interval $L$ or $R$
  respectively.  Since the exit measure is non-atomic almost surely, the
  event
  \[
  E_n=\{\min(p_L^n,p_R^n) > 1/3 \}
  \]
  has positive probability (in fact, it is not hard to see that each of the
  random variables $p^n_L$ is uniformly distributed on $[0,1]$ so that
  $E_n$ has probability $1/3$).  Moreover, the value of $p^n_L$ does not
  depend on the choice of circle packing and is therefore a function of
  $(G,(X_{n+k})_{k \in \Z})$.  By the stationarity and ergodicity
  of $(G,(X_n)_{n \in \Z})$, the events $E_n$ happen infinitely often
  almost surely (see \cref{fig:poisson}).

  \begin{figure}
    \begin{center}
    \includegraphics[width=0.8\textwidth]{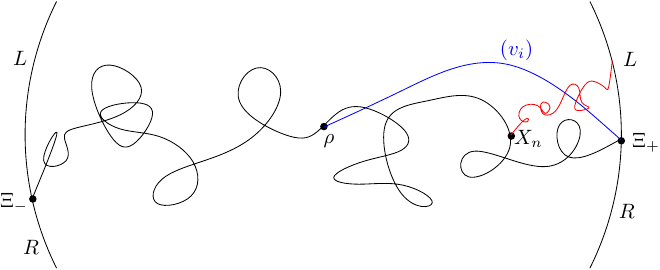}
    \end{center}
    \caption{For infinitely many $n$, a new random walk (red) started from
      $X_n$ has probability at
      least $1/3$ of hitting each of $L$ and $R$, and probability at least
      $1/4$ of hitting the path $(v_i)$ (blue).}
    \label{fig:poisson}
  \end{figure}

  Now condition on $\Xi_+ \in B$.  Since the exit measure of $B$ is
  positive, the events $E_n$ still happen infinitely often almost surely
  after conditioning.  Let $(v_i)_{i \geq 0}$ be a path from $\rho$ in $G$
  such that $z(v_i) \to \Xi_+ \in B$ and $h(v_i) \to 1$.  In
  particular,
  \[
  \inf\{h(v_i) : i\geq1\} > 0.
  \]
  The path $(\dots,X_{-2},X_{-1},\rho,v_1,v_2,\dots)$ disconnects $X_n$
  from at least one of the intervals $L$ or $R$ and so
  \begin{equation}
    P^G_{X_n}(\text{hit } \{\dots,X_{-1},\rho,v_1,\dots\})
    \geq \min(p^n_L,p^n_R)
  \end{equation}
  which is greater than $1/3$ infinitely often almost surely.  We stress
  that the expression refers to the probability that an independent random
  walk started from $X_n$ hits the path $(\dots,X_{-1},\rho,v_1,\dots)$,
  and that this bound holds trivially if $X_n$ is on the path $(v_m)$.  By
  \cref{lem:Poissonlem},
  \begin{equation}
    P^G_{X_n}(\text{hit } \{\dots,X_{-1},\rho\}) \xrightarrow{a.s.} 0,
  \end{equation}
  and hence $P^G_{X_n}(\text{hit } \{v_1,v_2,\dots\}) > 1/4$ infinitely
  often almost surely (see \cref{fig:poisson}).  Note that, since the
  choice of $v_i$ could depend on the whole trajectory of $X$, we have
  \emph{not} shown that $X$ hits the path $(v_i)$ infinitely often.
  % If we liked, we could choose a path $(v_i)$ to ensure this does not
  % happen, for example the path corresponding to the right hand outer
  % vertex boundary of $X$.
  Nevertheless, by \eqref{eq:ostcor}, almost surely infinitely often
  \begin{equation}
    h(X_n) > \frac{1}{4} \inf\{h(v_i) : i\geq1\} > 0.
  \end{equation}
  By L\'{e}vy's 0-1 law, $\lim_{n \to \infty} h(X_n) = 1$ almost surely as
  desired.
\end{proof}

%%%%%%%%%%%%%%%%%%%%%%%%%%%%%%%%%%%%%%%%%%%%%%%%%%%%%%%%%%%%%%%%%%%%%%%%%
\section{Hyperbolic speed and decay of radii}
\label{sec:exp_decay}

We now use the fact that the exit measure is almost surely non-atomic to
strengthen \cref{lem:exp_decay} and deduce that the limit rate of decay of
the Euclidean radii along the random walk exists.  The key idea is to use a
circle packing in the upper half-plane normalised by the limits of two
independent random walks.

Fix some circle packing $\cC$ in $\D$, so that, by \cref{thm:boundary}, the
limit points $\Xi_\pm=\lim_{n \to \pm\infty}z(X_n)$ exist and are
distinct almost surely.  Let $\Phi_X$ be a M\"obius transformation that
maps $\D$ to the upper half-plane $\H$ and sends $\Xi_+$ to 0 and
$\Xi_-$ to $\infty$.  We consider the upper half-plane packing
$\widehat\cC = \Phi_X(\cC)$.

Similarly to the proof of non-atomicity in \cref{sec:exitmeasure}, we now
have two boundary points $0$ and $\infty$ fixed by the graph $G$ and the path
$(X_n)$, so that the resulting circle packing is unique up to scaling.
Now, however, the packing depends on both $G$ and the random walk, so that
this new situation is not paradoxical (as it was in \cref{sec:exitmeasure}
where we ruled out the possibility that the exit measure has a single
atom).

\begin{proof}[Proof of \cref{thm:exp_decay}]
  We prove the equivalent statement for $(G,\rho)$ reversible with
  $\E[\deg(\rho)]<\infty$, and may assume that $(G,\rho)$ is ergodic.  We
  fix a circle packing $\widehat\cC=\Phi_X(\cC)$ in $\H$ as above, with the
  doubly infinite random walk from $\infty$ to $0$.  Let $\hat{r}(v)$ be
  the Euclidean radius of the circle corresponding to $v$ in $\widehat\cC$.
  The ratio of radii $\hat{r}(X_n)/\hat{r}(X_{n-1})$ does not depend on the
  choice of $\widehat\cC$, so these ratios form a stationary ergodic
  sequence.  By the Sharp Ring Lemma, $\E \big[
  |\log(\hat{r}(X_1)/\hat{r}(\rho))| \big] \leq C \E\big[\deg(\rho)\big] <
  \infty$, so that the Ergodic Theorem implies that
  \begin{equation}
    -\frac{1}{n}\log \frac{\hat{r}(X_n)}{\hat{r}(\rho)} =
    -\frac{1}{n}\sum_1^n \log \frac{\hat{r}(X_i)}{\hat{r}(X_{i-1})}
    \xrightarrow[n\to\infty]{a.s.}
    - \E\left[\log \frac{\hat{r}(X_1)}{\hat{r}(\rho)}\right].
  \end{equation}
%  It follows that for any choice of $\hat\cC$ we have $\lim \frac{-\log
%    \hat{r}(X_n)}{n} = -\E \left[\log \frac{\hat{r}(X_1)}{\hat{r}(\rho)}
%  \right]$.

  Now, since $\widehat\cC$ is the image of $\cC$ through the M\"obius map
  $\Phi_X$, and since $\Phi_X$ is conformal at $\Xi_+$,
  \begin{equation}
    \frac{\hat{r}(X_n)}{r(X_n)} \to
    \big|\Phi_X'\left(\Xi_+\right)\big| > 0.
  \end{equation}
  Therefore
  \begin{equation}
    \lim \frac{-\log r(X_n)}{n} = \E\left[- \log
      \frac{\hat{r}(X_1)}{\hat{r}(\rho)} \right]
  \end{equation}
  and by \cref{lem:exp_decay} this limit must be positive.  This
  establishes the rate of decay of the radii.

  \medskip

  Next, we relate this to the distance of $z(X_n)$ from $\partial\D$.  By
  the triangle inequality, $1-|z(X_n)|$ is at most the length of the path
  formed by drawing straight lines between the Euclidean centres of the
  circles along the random walk path starting at $X_n$:
  \[
  1 - |z_h(X_n)| \leq 1-|z(X_n)| \leq \sum_{i\geq n} 2r(X_i).
  \]
  Since the radii decay exponentially, taking the limits of the logarithms,
  \begin{equation}
    \label{eq:speed1}
    \liminf \frac{-\log\big(1-|z_h(X_n)|\big)}{n}
    \geq \lim \frac{-\log r(X_n)}{n}.
  \end{equation}

  To get a corresponding upper bound, note that, since every circle
  neighbouring $X_n$ is contained in the open unit disc, $1-|z_h(X_n)|$ is at
  least the radius of the smallest neighbour of $X_n$.  Applying the Sharp
  Ring Lemma, we have
  \[
  1-|z_h(X_n)| \geq r(X_n) \exp(-C\deg(X_n)).
  \]
  Taking logarithms and passing to the limit,
  \begin{align}
    \limsup \frac{-\log\big(1-|z_h(X_n)|\big)}{n}
    &\leq \lim\frac{-\log r(X_n)}{n} + \lim\frac{C\deg(X_n)}{n} \nonumber\\
    &= \lim \frac{-\log r(X_n)}{n}, \label{eq:speed2}
  \end{align}
  where the almost sure limit $\deg(X_n)/n \to 0$ follows from
  $\E[\deg(\rho)]<\infty$ and Borel-Cantelli.  Combining \eqref{eq:speed1}
  and \eqref{eq:speed2} gives the  almost sure limit
  \[
  \lim \frac{-\log\big(1-|z_h(X_n)|\big)}{n}= \lim \frac{-\log r(X_n)}{n}.
  \]
  Finally, to relate this to the speed in the hyperbolic metric, recall that
  distances from the origin in the hyperbolic metric are given by
  \[ d_{hyp}(0,z) = 2 \tanh^{-1}|z|\]
  and hence
  \begin{align*}
    \lim \frac{1}{n} d_{hyp}(z_h(\rho),z_h(X_n))
    &= \lim \frac{1}{n} d_{hyp}(0,z_h(X_n)) \\
    & = \lim\frac{2}{n} \tanh^{-1}|z_h(X_n)| \\
    & = \lim -\frac{1}{n}\log (1-|z_h(X_n)|).  \qedhere
  \end{align*}
\end{proof}

%\note{if the expected edge length is finite, the sub-additive ergodic
%  theorem gives a.s. existence of the hyperbolic limit speed. This
%  eliminates the extra moment assumption on $\deg(\rho)$ in this part of
%  the proof.}

%%%%%%%%%%%%%%%%%%%%%%%%%%%%%%%%%%%%%%%%%%%%%%%%%%%%%%%%%%%%%%%%%
\section{Extensions}
\label{sec:extensions}

We now discuss two basic extensions of our main results beyond simple
triangulations.  These are to weighted and to non-simple triangulations.
The latter are of particular interest since the PSHT is not simple.
Some of our results hold for much more general planar maps, which are
treated in \cite{unimodular2}.

\paragraph{Weighted networks.}
%One natural extension of our results concerns weighted triangulations.
Suppose $(G,\rho,w)$ is a unimodular random rooted weighted
triangulation.  As in the unweighted case, if $\E[w(\rho)]$ is finite then
biasing by $w(\rho)$ gives an equivalent random rooted weighted
triangulation which is reversible for the weighted simple random walk
\cite[Theorem 4.1]{AL07}.  Our arguments generalise with no change to
recover all our main results in the weighted setting provided the following
conditions are satisfied.
\begin{enumerate}
\item $\E[w(\rho)]<\infty$.  This allows us to bias to get a reversible
  random rooted weighted triangulation.
\item $\E[w(\rho)\deg(\rho)]<\infty$.  After biasing by $w(\rho)$, the
  expected degree is finite, allowing us to apply the Ring Lemma together
  with the Ergodic Theorem as in the proofs of \cref{lem:exp_decay} and
  \cref{thm:exp_decay}.
\item A version of \cref{thm:BLS} holds. That is, there exists a
  percolation $\omega$ such that the induced network $\bar \omega$ has
  positive Cheeger constant almost surely. Two natural situations in which
  this occurs are
\begin{enumerate}
\item when all the weights are non-zero almost surely.
  In this situation, we may adapt the proof of \cref{thm:BLS} by first
  deleting all edges of weight less than $1/M$ and all vertices of total
  weight greater than $M$ before continuing the construction as before.
\item when the subgraph formed by the edges of non-zero weight is connected
  and is itself invariantly non-amenable.  This occurs when we circle pack
  planar maps that are not triangulations by adding edges of weight $0$ in
  non-triangular faces to triangulate them.
\end{enumerate}
\end{enumerate}

\paragraph{Non-simple triangulations.}
Suppose $G$ is a one-ended planar map.  The endpoints of any double edge or
loop in $G$ disconnect $G$ into connected components exactly one of which
is infinite.  The \textbf{simple core} of $G$, denoted $\core(G)$, is
defined by deleting the finite component contained within each double edge
or loop of $G$ before gluing the double edges together or deleting the loop
as appropriate.  See \cref{fig:core} for an example, and \cite{UIPT1} for a
more detailed description.  When $G$ is a triangulation, so is its core.
The core can be seen as a subgraph of $G$, with some vertices removed, and
multiple edges replaced by a single edge.  The induced random walk on the
core, is therefore a random walk on a weighted simple triangulation.

\begin{figure}
  \centering
  \includegraphics[width=.8\textwidth]{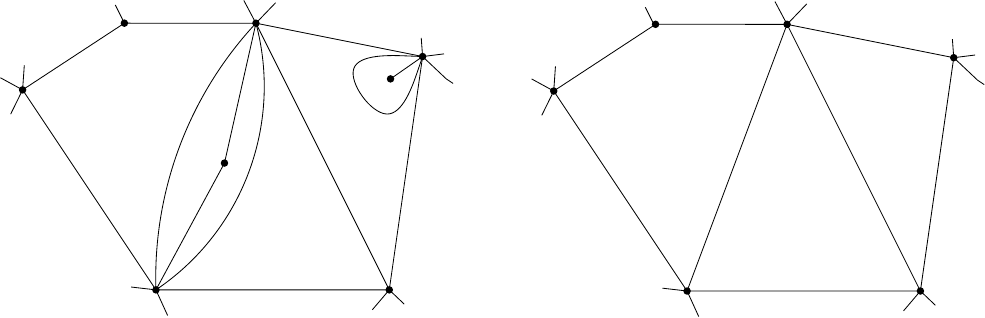}
  \caption{Extracting the core of a non-simple map. Left: part of a
    map. Right: corresponding part of its core.}
  \label{fig:core}
\end{figure}

In general, it is possible that all of $G$ is deleted by this procedure,
but in this case there are infinitely many disjoint vertex cut-sets of size
2 separating each vertex from infinity, implying that $G$ is VEL parabolic
and hence invariantly amenable.  When $G$ is invariantly non-amenable, the
conclusions of \cref{thm:boundary} hold with the necessary modifications.

\begin{thm}\label{thm:coreboundary}
  Let $(G,\rho)$ be an invariantly non-amenable, one-ended, unimodular
  random rooted planar triangulation with $\E[\deg^2(\rho)]<\infty$.  Then
  $\core(G)$ is CP hyperbolic.  Let
  $\cC$ be a circle packing of $\core(G)$ in $\D$, and let $(Y_n)_{n\in\N}$ be
  the induced random walk on $\core(G)$.  The following hold conditional on
  $(G,\rho)$ almost surely:
  \begin{enumerate}\itemsep0em
  \item $z(Y_n)$ and $z_h(Y_n)$ both converge to a (random) point denoted $\Xi
    \in\partial \D$,
  \item The law of $\Xi$ has full support and no atoms.
  \item $\partial \D$ is a realisation of the Poisson boundary of $G$. That
    is, for every bounded harmonic function $h$ on $G$ there exists a
    bounded measurable function $g: \partial \D \to \R$ such that
    \[
    h(v) = E^G_v \big[ g(\Xi)\big].
    \]
  \end{enumerate}
\end{thm}

Since the additional components needed to prove this are straightforward,
we omit some of the details.

\begin{proof}[Sketch of proof]
  First, $(\core(G),\rho)$ is unimodular when sampled conditional on
  $\rho \in \core(G)$: essentially, a mass transport on $\core(G)$ gives a
  mass transport on $G$ which is $0$ for all deleted vertices. The mass
  transport principle for $G$ implies the principle for $\core(G)$.

  Second, $\core(G)$ is CP hyperbolic. Since $(G,\rho)$ is invariantly
  non-amenable, it is VEL hyperbolic (see the proof of
  \cref{thm:classification}).  Because the infimum over paths in the
  definition of the vertex extremal length is the same as the infimum over
  paths in the core, the vertex extremal length from $v\in \core(G)$ to
  $\infty$ is the same in $G$ and $\core(G)$. (Alternatively, one could
  deduce non-amenability of $\core(G)$ from non-amenability of $G$, and
  apply \cref{thm:classification}.)

  Now, since $\core(G)$ is a weighted CP hyperbolic unimodular simple
  triangulation (and the second moment of the degree of the root is
  finite), by \cref{thm:boundary} the random walk on $\core(G)$ converges
  to a point in the boundary, the exit measure has full support and no
  atoms, and $\partial\D$ is a realisation of the Poisson boundary of
  $\core(G)$.

  Finally, by the Optional Stopping Theorem, the bounded harmonic functions
  on $\core(G)$ are in one-to-one correspondence with the
  bounded harmonic functions on $G$ by restriction and extension:
  \begin{equation*}
    h_G \mapsto h_{\core(G)} = h_G|_{\core(G)}, \quad
    h_{\core(G)} \mapsto h_G(v) = E_v^G[h_{\core(G)}(X_{N_0})].
  \end{equation*}
  Thus, the realisation of $\partial \D$ as the Poisson boundary of
  $\core(G)$ extends to $G$.
\end{proof}

%%%%%%%%%%%%%%%%%%%%%%%%%%%%%%%%%%%%%%%%%%%%%%%%%%%%%%%%%%%%%%%%%%%%%%%%%%%
\section{Open Problems}
\label{sec:problems}

\begin{problem}
  Can the identification of the Poisson and geometric boundaries be
  strengthened to an identification of the Martin boundary? This was
  done in \cite{ABGN14} for CP hyperbolic triangulations with bounded
  degrees.
  Specifically, we believe the following.
\end{problem}

\begin{conjecture}
  Let $(G,\rho)$ be an infinite simple, one-ended, CP hyperbolic unimodular
  random rooted planar triangulation with $\E[\deg(\rho)]<\infty$, and let
  $\cC$ be a circle packing of $G$ in the unit disc. Then almost surely for
  every point $\xi\in \partial \D$ there exists a unique positive harmonic
  function $h_\xi$ on $G$ such that $h_\xi(\rho)=1$ and $h_\xi$ is bounded
  on $\{v : |z(v)-\xi|\geq \eps\}$ for every $\eps>0$. Moreover, the
  function $\xi\mapsto h_\xi$ almost surely extends to a homeomorphism from
  $z(V)\cup\partial \D$ to the Martin compactification of $G$.
  % Then the Martin
  % boundary of $G$ is almost surely homeomorphic to either
  % \begin{enumerate} \itemsep0em
  % \item A point if $(G,\rho)$ is invariantly amenable, or
  % \item The circle $\partial \D$ if $M$ is invariantly non-amenable and
    % one-ended.
  % \item The space of ends of $M$ if $M$ is invariantly non-amenable and
    % infinitely-ended.
  % \end{enumerate}
  % In case (2), the homeomorphism should arise from circle packing a
  % suitable triangulation associated to $M$.
\end{conjecture}

% \begin{problem}[Fatou-type Theorem]
% Does a Fatou-type Theorem hold under the assumptions of \cref{thm:boundary}? That is, whenever $f:\partial\D\to\R$ is a bounded function with harmonic extension $h=E_v^G[f(\Xi)]$, does it hold almost surely that for almost every $\xi\in\partial \D$ (with respect to the exit measure), and every sequennce of vertices of $G$ such that $z(v_n)$ converges non-tangentially to $\xi$, we have
% \[\lim_{n\to\infty}h(v_n)=f(\xi)?\]
% \end{problem}

\begin{problem}[H\"older continuity of the exit measures]
  In the setting of \cref{thm:boundary}, do there
  exist positive constants $c$ and $C$ such that
  \[
  P_\rho^G( \Xi \in I) \leq C |I|^c
  \]
  for every interval $I \subset \partial
  \D$?
\end{problem}

\begin{problem}[Dirichlet energy of $z$]
  In the bounded degree case, by applying the main theorem of \cite{ALP99}, convergence to the boundary may be shown by
  observing that the Dirichlet energy of the centres function $z$ is finite:
  \[
  \mathcal{E}(z) = \sum_{u \sim v} (z(u)-z(v))^2
  \leq \sum 2\deg(v)r(v)^2 \leq 2\max\{\deg(v)\}.
  \]
  Is the Dirichlet energy of $z$ almost surely finite for a unimodular
  random rooted CP hyperbolic triangulation? This may provide a route to
  weakening the moment assumption in our results.
\end{problem}

\begin{problem}[Other embeddings]
  How does the canonical embedding of the Poisson-Delaunay triangulation
  differ from the embedding given by the circle packing? Is there a circle
  packing so that $d_{hyp}(v, z_h(v))$ is stationary?

  The \textbf{conformal embedding} of a triangulation is defined by
  forming a Riemann surface by gluing equilateral triangles according to
  the combinatorics of the triangulation before mapping the resulting surface
  conformally to $\D$ or $\C$.  Is it possible to control the large scale
  distortion between the conformal embedding and the circle packing?  In
  general the answer is no, but in the unimodular case there is hope.

  Regardless of the answer to this question, our methods should extend
  without too much difficulty to establish analogues of \cref{thm:boundary}
  for these other embeddings, the main obstacle being to show almost sure
  convergence of the random walk to a point in the boundary $\partial\D$.
\end{problem}

\begin{problem}
  Reduce the moment assumption on $\deg(\rho)$ in
  \cref{thm:boundary,thm:exp_decay}.  Finite expectation is needed to
  switch to a reversible distribution on rooted maps, but perhaps the
  second moment is not needed.
\end{problem}
%
%\begin{problem}\label{Q:unimod}
%  Is every unimodular hyperbolic planar triangulation a Benjamini-Schramm
%  limit of finite (random) maps? If so, the finite maps will necessarily
%  have genus linear in their size.  This is a special case of particular
%  interest of the more general question of whether every unimodular graph
%  is a Benjamini-Schramm limit of finite graphs (see \cite{AL07}).
%\end{problem}
%
%\note{comment on a positive answer in the dichotomy paper?}
%% Answer: no, we already did so in the introduction.

%%%%%%%%%%%%%%%%%%%%%%%%%%%%%%%%%%%%%%%%%%%%%%%%%%%%%%%%%%%%%%%%%%
\subsection*{Acknowledgments}
OA is supported in part by NSERC.  AN is supported by the Israel Science
Foundation grant 1207/15 as well as NSERC and NSF grants.  GR is supported
in part by the Engineering and Physical Sciences Research Council under
grant EP/103372X/1.

All circle packings above were generated using Ken Stephenson's CirclePack
software \cite{CP}.  We thank Ken for his assistance using this software and
for useful conversations.  We also thank the referee for their comments and
suggestions.

{\small
\bibliographystyle{abbrv}
\bibliography{unimodular}
}

\end{document}